%% file: CA.tex
\documentclass{ceb-l}

\usepackage{graphicx}
\usepackage{pstricks}
\usepackage[all]{xy}

\DeclareMathOperator{\Diff}{Diff}

\newcommand{\ko}{\: , \;}

\newcommand{\bt}{\bullet}

\newcommand{\del}{\partial}
\newcommand{\T}{\mathcal{T}}
\newcommand{\TTT}{\widetilde{\mathcal{T}}}
\newcommand{\zg}{\gamma}
\newcommand{\za}{\alpha}
\newcommand{\BB}{\mathbb{B}}
\newcommand{\A}{\mathcal{A}}
\newcommand{\PP}{\mathbb{P}}
\newcommand{\x}{\mathbf{x}}
\newcommand{\xx}{\mathbf{x}}
\newcommand{\F}{\mathcal{F}}
\newcommand{\C}{\mathbb{C}}
\newcommand{\R}{\mathbb{R}}
\newcommand{\Z}{\mathbb{Z}}
\newcommand{\CC}{\mathcal{C}}
\newcommand{\overunder}[2]{
\!\begin{array}{c}
\scriptstyle{#1}\\[-.1in]
-\!\!\!-\!\!\!-\\[-.1in]
\scriptstyle{#2}
\end{array}
\!
}

\def\Xcal{\mathcal{X}}
\def\Acal{\mathcal{A}}
\def\Fcal{\mathcal{F}}
\def\Q{\mathbb{Q}}
\def\QQ{\mathbb{Q}}
\def\ZZ{\mathbb{Z}}
\def\TT{\mathbb{T}}
\def\Trop{\operatorname{Trop}}
\def\yy{\mathbf{y}}
\def\vv{\mathbf{v}}

\newtheorem{theorem}{Theorem}[section]
\newtheorem{lemma}[theorem]{Lemma}

\theoremstyle{definition}
\newtheorem{definition}[theorem]{Definition}
\newtheorem{example}[theorem]{Example}
\newtheorem{exercise}[theorem]{Exercise}
\newtheorem{conjecture}[theorem]{Conjecture}
\newtheorem{prop}[theorem]{Proposition}

\theoremstyle{remark}
\newtheorem{remark}[theorem]{Remark}

\numberwithin{equation}{section}

\begin{document}

\title{Cluster algebras: an introduction}


\author{Lauren K. Williams}
\address{Department of Mathematics, University of California,
Berkeley, CA 94720}
\curraddr{}
\email{williams@math.berkeley.edu}
\thanks{The author is partially supported by a Sloan Fellowship
and an NSF Career award.}


\date{}

\dedicatory{Dedicated to Andrei Zelevinsky on the
occasion of his $60$th birthday}

\begin{abstract}
Cluster algebras are commutative rings with a set of distinguished
generators having a remarkable combinatorial structure. They were introduced
by Fomin and Zelevinsky in 2000 in the context of Lie theory, but
have since appeared in many other contexts, from Poisson geometry to
triangulations of surfaces and Teichm\"uller theory.
In this expository paper we give an introduction to cluster algebras,
and illustrate how this framework naturally arises in Teichm\"uller theory.
We then sketch how the theory of cluster algebras
led to a proof of the Zamolodchikov periodicity conjecture in
mathematical physics.
\end{abstract}

\maketitle
\setcounter{tocdepth}{1}
\tableofcontents

\section{Introduction}\label{intro}

Cluster algebras were conceived by Fomin and Zelevinsky \cite{FZ1}
in the spring of 2000 as a tool for studying total positivity and 
dual canonical bases in Lie theory.  However, the theory of cluster algebras
has since taken on a life of its own, as connections and applications have
been discovered to diverse areas of mathematics including
quiver representations, Teichm\"uller theory,
tropical geometry, integrable systems, and Poisson geometry.  

In brief, a {\it cluster algebra} $\A$ of rank $n$ is a subring of an
{\it ambient field} $\F$ of rational functions
in $n$ variables.  Unlike ``most" commutative rings, a cluster
algebra is not presented at the outset via a complete set of 
generators and relations.  Instead, from the initial data of 
a \emph{seed} -- which includes $n$ distinguished 
generators called \emph{cluster variables}
plus an \emph{exchange matrix} -- one uses an iterative procedure
called \emph{mutation} to produce  the rest of the 
cluster variables. 
In particular, each new cluster variable is a rational expression
in the initial cluster variables.  The cluster algebra is then defined to be 
the subring of $\F$ generated by all cluster variables.

The set of cluster variables has a remarkable combinatorial structure:
this set is a union of overlapping
algebraically independent $n$-subsets of $\F$ called {\it clusters},
which together have the structure of a simplicial complex 
called the {\it cluster
complex}.  
The clusters are related to each other by birational transformations
of the following kind: for every cluster $\x$
and every cluster variable $x\in \x$, there is another cluster
$\x' = (\x - \{x\}) \cup \{x'\}$, with the new cluster variable
$x'$ determined by an {\it exchange relation} of the form
\vspace{-.2cm}
\begin{equation*}
x x' = y^+ M^+ + y^- M^-.
\end{equation*}
Here $y^+$ and $y^-$ 
lie in a {\it coefficient semifield} $\PP$,
while $M^+$ and $M^-$ are  monomials in the elements of
$\mathbf x - \{x\}$.  
In the most general class of cluster algebras, 
there are two dynamics at play in the exchange relations:\
that of the monomials, and that of the coefficients, both of 
which are encoded in the exchange matrix. 

The aim of this article is threefold:
to give an elementary introduction to the theory of cluster
algebras;
to illustrate how the framework of cluster algebras
naturally arises in diverse areas of mathematics, in particular
Teichm\"uller theory;
and to illustrate how the theory of cluster algebras has been
an effective tool for solving outstanding conjectures,
in particular the Zamolodchikov periodicity conjecture from
mathematical physics.

To this end, in Section \ref{sect cluster algebras} we introduce
the notion of cluster algebra, beginning with the
simple but somewhat restrictive definition of a cluster algebra
defined by a quiver.  After giving a detailed example
(the \emph{type A cluster algebra}, and its realization
as the coordinate ring of the Grassmannian $Gr_{2,d}$), we give 
a more general definition of cluster algebra, in which both the cluster
variables and coefficient variables have their own dynamics. 

In Section \ref{sec:Teichmuller} we explain how cluster
algebras had appeared implicitly in Teichm\"uller theory, long before the 
introduction of cluster algebras themselves.  We start by associating
 a cluster algebra to any \emph{bordered surface with
marked points},
following
work of Fock-Goncharov \cite{FG1}, Gekhtman-Shapiro-Vainshtein \cite{GSV},
and Fomin-Shapiro-Thurston \cite{FST}.
This construction specializes to the type A example
from Section \ref{sect cluster algebras} when the surface is 
a disk with marked points on the boundary. 
We then explain how a cluster algebra from a bordered surface is
related to the decorated 
Teichm\"uller space of the corresponding \emph{cusped surface}.  
Finally we briefly 
discuss the Teichm\"uller space of a surface with oriented
geodesic boundary and two related spaces of laminations, and how 
natural coordinate systems  on these spaces are related to 
cluster algebras.

In Section \ref{sec:Z} we discuss Zamolodchikov's 
\emph{periodicity conjecture}
for \emph{Y-systems} \cite{Z}. 
Although this conjecture arose from Zamolodchikov's study of the 
thermodynamic Bethe ansatz in mathematical physics,
Fomin-Zelevinsky realized that it could be reformulated
in terms of the dynamics of coefficient variables
in a cluster algebra \cite{FZY}.
We then discuss how Fomin-Zelevinsky used 
fundamental structural results for finite type
cluster algebras to prove the periodicity 
conjecture for Dynkin diagrams \cite{FZY}, and how Keller used 
deep results from the categorification of cluster algebras to prove 
the corresponding conjecture for pairs of Dynkin diagrams \cite{Keller1, Keller2}.




\textsc{Acknowledgements:}
This paper was written to accompany my talk at the 
Current Events Bulletin Session at the Joint Mathematics Meetings
in San Diego, in January 2013;  I would like to thank the organizers
for the impetus to prepare this document.
I gratefully acknowledge the hospitality of MSRI 
during the Fall 2012 program on cluster algebras,
which provided an ideal environment for writing this paper.
In addition,
I am indebted to Bernhard Keller, Tomoki Nakanishi, and Dylan Thurston 
for useful conversations, and to  Keller for the use of several
figures.  
Finally, I am grateful to an anonymous referee for insightful comments.

\section{What is a cluster algebra?}\label{sect cluster algebras}

In this section we will define the notion of cluster algebra,
first introduced by Fomin and Zelevinsky in \cite{FZ1}.
For the purpose of acquainting the reader with the basic notions,
in Section \ref{sec:first} 
we will give the simple but somewhat restrictive definition
of a \emph{cluster algebra defined by a quiver}, also 
called a \emph{skew-symmetric cluster algebra of geometric type}.  
We will give a detailed example in Section \ref{sec:A}, and 
then present a more general definition of cluster algebra 
in Section \ref{sec:generaldef}.

\subsection{Cluster algebras from quivers}\label{sec:first}


\begin{definition}
[\emph{Quiver}]
A \emph{quiver} $Q$ is an oriented graph given by a set of 
vertices $Q_0$, a set of arrows $Q_1$, and two maps 
$s: Q_1 \to Q_0$ and $t: Q_1 \to Q_0$ taking an arrow to its source
and target, respectively.  
\end{definition}

A quiver $Q$ is \emph{finite} if the sets
$Q_0$ and $Q_1$ are finite.  
A \emph{loop} of a quiver is an arrow $\alpha$ whose
source and target coincide.  A \emph{$2$-cycle} of a quiver is a pair of distinct 
arrows $\beta$ and $\gamma$ such that $s(\beta) = t(\gamma)$
and $t(\beta) = s(\gamma)$.

\begin{definition}
[\emph{Quiver Mutation}]
Let $Q$ be a finite quiver without loops or $2$-cycles.
Let $k$ be a vertex of $Q$.  
Following \cite{FZ1}, we define the 
\emph{mutated quiver} $\mu_k(Q)$ as follows:
it has the same set of vertices as $Q$,
and its set of arrows is obtained by the following procedure:
\begin{enumerate}
\item for each subquiver $i \to k \to j$, add a new arrow $i \to j$;
\item reverse all allows with source or target $k$;
\item remove the arrows in a maximal set of pairwise
disjoint $2$-cycles.
\end{enumerate}
\end{definition}

\begin{exercise}
Mutation is an involution, that is,
$\mu_k^2(Q) = Q$ for each vertex $k$.  
\end{exercise}
Figure \ref{fig:mutation} shows two quivers which are obtained from 
each other by mutating at the vertex $1$.
We say that two quivers $Q$ and $Q'$ are {\it mutation-equivalent}
if one can get from $Q$ to $Q'$ by a sequence of mutations.

\begin{figure}[h]
\centering
\includegraphics[height=.8in]{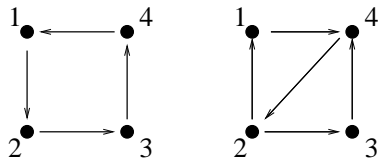}
\caption{Two mutation-equivalent quivers.}
\label{fig:mutation}
\end{figure}

\begin{definition}\label{rem:Bmatrix}
Let $Q$ be a finite quiver with no loops or $2$-cycles and whose
vertices are labeled $1,2,\dots,m$.  Then we may encode
$Q$  by an $m \times m$ 
skew-symmetric \emph{exchange matrix}  $B(Q) = (b_{ij})$ where 
$b_{ij}=-b_{ji} = \ell$ whenever there are $\ell$ arrows from vertex $i$ to vertex $j$.  We call $B(Q)$ the \emph{signed adjacency matrix} of the quiver.
\end{definition}

\begin{exercise}
Check that when one encodes a quiver $Q$ by a matrix
as in Definition \ref{rem:Bmatrix},
the matrix $B(\mu_k(Q))=(b'_{ij})$ 
is again an $m \times m$ skew-symmetric matrix,
whose entries are given by
\begin{equation}
\label{eq:matrix-mutation0}
b'_{ij} =
\begin{cases}
-b_{ij} & \text{if $i=k$ or $j=k$;} \\[.05in]
b_{ij}+b_{ik}b_{kj} & \text{if $b_{ik}>0$ and $b_{kj}>0$;}\\[.05in]
b_{ij}-b_{ik}b_{kj} & \text{if $b_{ik}<0$ and $b_{kj}<0$;}\\[.05in]
b_{ij} & \text{otherwise.}
\end{cases}
\end{equation}
\end{exercise}


\begin{definition}
[\emph{Labeled seeds}]
\label{def:seed0}
Choose $m\geq n$ positive integers.
Let $\Fcal$ be an \emph{ambient field} 
of rational functions 
in $n$ independent
variables
over
$\Q(x_{n+1},\dots,x_m)$. 
A \emph{labeled seed} in~$\Fcal$ is
a pair $(\xx, Q)$, where
\begin{itemize}
\item
$\xx = (x_1, \dots, x_m)$ forms a free generating 
set for 
$\Fcal$,
and
\item
$Q$ is a quiver on vertices
$1, 2, \dots,n, n+1, \dots, m$,
whose vertices $1,2, \dots, n$ are called
\emph{mutable}, and whose vertices $n+1,\dots, m$ are called \emph{frozen}.
\end{itemize}
We refer to~$\xx$ as the (labeled)
\emph{extended cluster} of a labeled seed $(\xx, Q)$.
The variables $\{x_1,\dots,x_n\}$ are called \emph{cluster
variables}, and the variables $c=\{x_{n+1},\dots,x_m\}$ are called
\emph{frozen} or \emph{coefficient variables}.
\end{definition}



\begin{definition}
[\emph{Seed mutations}]
\label{def:seed-mutation0}
Let $(\xx, Q)$ be a labeled seed in $\Fcal$,
and let $k \in \{1,\dots,n\}$.
The \emph{seed mutation} $\mu_k$ in direction~$k$ transforms
$(\xx, Q)$ into the labeled seed
$\mu_k(\xx,  Q)=(\xx', \mu_k(Q))$, where the cluster 
$\xx'=(x'_1,\dots,x'_m)$ is defined as follows:
$x_j'=x_j$ for $j\neq k$,
whereas $x'_k \in \Fcal$ is determined
by the \emph{exchange relation}
\begin{equation}
\label{exchange relation0}
x'_k\ x_k = 
 \ \prod_{\substack{\alpha\in Q_1 \\ s(\alpha)=k}} x_{t(\alpha)}
+ \ \prod_{\substack{\alpha\in Q_1 \\ t(\alpha)=k}} x_{s(\alpha)} \, .
\end{equation}
\end{definition}

\begin{remark}
Note that arrows between two frozen vertices of a quiver do not 
affect seed mutation (they do not affect the mutated quiver
or the exchange relation).  For that reason, one may omit
arrows between two frozen vertices.  Correspondingly,
when one represents a quiver by a matrix, one often omits 
the data corresponding to such arrows.  The resulting matrix 
$B$ is hence an $m \times n$ matrix rather than an $m \times m$ one.
\end{remark}

\begin{example}
Let $Q$ be the quiver on two vertices $1$ and $2$ with a single
arrow from $1$ to $2$.  Let $((x_1,x_2), Q)$ be an initial seed.  Then 
if we perform seed mutations in directions $1$, $2$, $1$, $2$, and $1$, 
we get the sequence of labeled seeds shown in Figure \ref{fig:seeds}.
Note that up to relabeling of the vertices
of the quiver, the initial seed and final seed coincide.
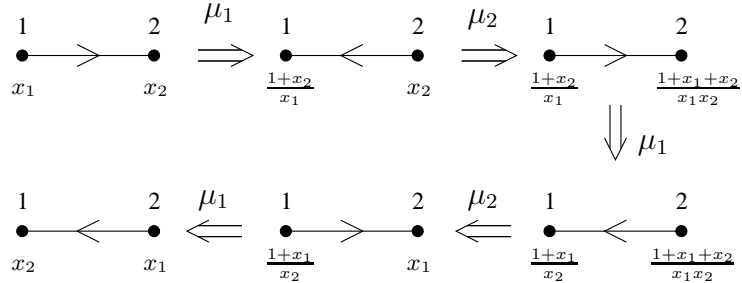
\begin{figure} \begin{center}
\input{A2.pstex_t}
\end{center}
\caption{Seeds and seed mutations in type $A_2$.}
\label{fig:seeds}
\end{figure}

\end{example}

\begin{definition}
[\emph{Patterns}]
\label{def:patterns0}
Consider the \emph{$n$-regular tree}~$\TT_n$
whose edges are labeled by the numbers $1, \dots, n$,
so that the $n$ edges emanating from each vertex receive
different labels.
A \emph{cluster pattern}  is an assignment
of a labeled seed $\Sigma_t=(\xx_t, Q_t)$
to every vertex $t \in \TT_n$, such that the seeds assigned to the
endpoints of any edge $t \overunder{k}{} t'$ are obtained from each
other by the seed mutation in direction~$k$.
The components of
$\xx_t$ are written as $\xx_t = (x_{1;t}\,,\dots,x_{n;t}).$
\end{definition}

Clearly, a cluster pattern  is uniquely determined
by an arbitrary  seed.

\begin{definition}
[\emph{Cluster algebra}]
\label{def:cluster-algebra0}
Given a cluster pattern, we denote
\begin{equation}
\label{eq:cluster-variables0}
\Xcal
= \bigcup_{t \in \TT_n} \xx_t
= \{ x_{i,t}\,:\, t \in \TT_n\,,\ 1\leq i\leq n \} \ ,
\end{equation}
the union of clusters of all the seeds in the pattern.
The elements $x_{i,t}\in \Xcal$ are called \emph{cluster variables}.
The 
\emph{cluster algebra} $\Acal$ associated with a
given pattern is the $\ZZ[c]$-subalgebra of the ambient field $\Fcal$
generated by all cluster variables: $\Acal = \ZZ[c] [\Xcal]$.
We denote $\Acal = \Acal(\xx,  Q)$, where
$(\xx,Q)$
is any seed in the underlying cluster pattern.
In this generality, 
$\Acal$ is called a \emph{cluster algebra from a quiver}, or a 
\emph{skew-symmetric cluster algebra of geometric type.}
We say that $\Acal$ has \emph{rank $n$} because each cluster contains 
$n$ cluster variables.
\end{definition}

\subsection{Example: the type $A$ cluster algebra}\label{sec:A}

In this section we will construct a cluster algebra using the 
combinatorics of triangulations of a $d$-gon  
(a polygon with $d$ vertices).  We will subsequently identify this
cluster algebra with the homogeneous coordinate ring of the Grassmannian
$Gr_{2,d}$ of $2$-planes in a $d$-dimensional vector space.

\begin{figure}[h]
\centering
\includegraphics[height=1.6in]{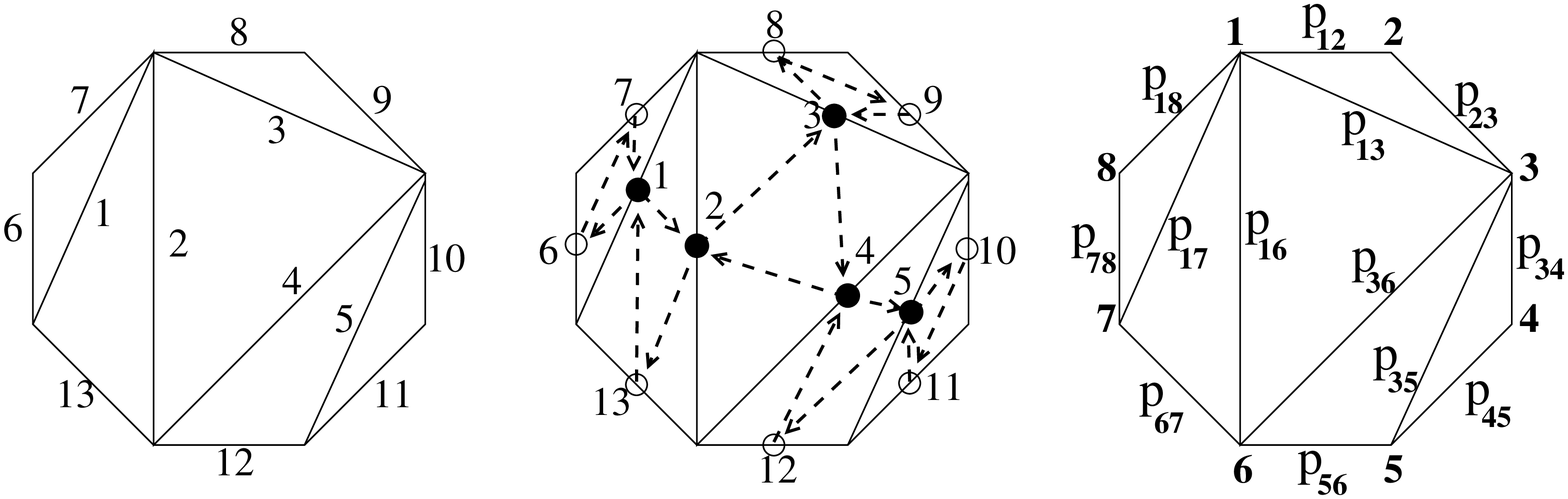}
\caption{A triangulation $T$ of an octagon, the quiver $Q(T)$,
and the labeling of $T$ by Pl\"ucker coordinates.}
\label{fig:octagon}
\end{figure}

\begin{definition}
[\emph{The quiver $Q(T)$}]
\label{def:QT}
Consider a $d$-gon $(d \geq 3$), and choose any triangulation $T$.
Label the $d-3$ diagonals of $T$ with the numbers $1,2,\dots,d-3$,
and label the $d$ sides of the polygon by the 
numbers $d-2, d-1, \dots,2d-3$.
Put a frozen vertex at the midpoint of each side of the polygon,
and put a mutable vertex at the midpoint of each diagonal of the polygon.
These $2d-3$ vertices are the vertices $Q_0(T)$ of $Q(T)$; 
label them according to the labeling of the diagonals and sides 
of the polygon.
Now within each triangle of $T$, inscribe a new triangle on the 
vertices $Q_0(T)$, whose edges
are oriented clockwise.  The edges of these inscribed triangles
comprise the set of arrows $Q_1(T)$ of $Q(T)$.
\end{definition}

See the left and middle of 
Figure \ref{fig:octagon} for an example of a triangulation
$T$ together with the corresponding quiver $Q(T)$.  
The frozen vertices
are indicated by hollow circles and the mutable vertices are indicated
by shaded circles.  The arrows of the quiver are indicated by dashed lines.

\begin{definition}
[\emph{The cluster algebra associated to a $d$-gon}]
Let $T$ be any triangulation of a $d$-gon, let
$m=2d-3$, and let $n=d-3$.  Set
$\xx=(x_1,\dots,x_{m})$.  Then 
$(\xx, Q(T))$ is a labeled seed and it determines a 
cluster algebra $\Acal(T) = \Acal(\xx,Q(T))$.
\end{definition}

\begin{remark}
The quiver $Q(T)$ depends on the choice of triangulation $T$.  However,
we will see 
in Proposition \ref{prop:ind}
that (the isomorphism class of) the cluster algebra $\Acal(T)$ does not depend on $T$,
only on the number $d$.
\end{remark}


\begin{definition}
[\emph{Flips}]
Consider a triangulation $T$ which contains a diagonal $t$.
Within $T$, the diagonal $t$ is the diagonal of some quadrilateral.
Then there is a new triangulation $T'$ which is obtained
by replacing the diagonal $t$ with the other diagonal of 
that quadrilateral.  This local move is called a \emph{flip}.
\end{definition}

Consider the graph whose vertex set is the set of triangulations 
of a $d$-gon, with an edge between two vertices whenever
the corresponding triangulations are related by a flip.
It is 
well-known that this ``flip-graph" is 
connected, and moreover, 
is the $1$-skeleton of a convex polytope called the 
associahedron.  See Figure \ref{fig:associahedron}
for a picture of the flip-graph of the hexagon.
\begin{figure}[h]
\centering
\includegraphics[height=4.7in]{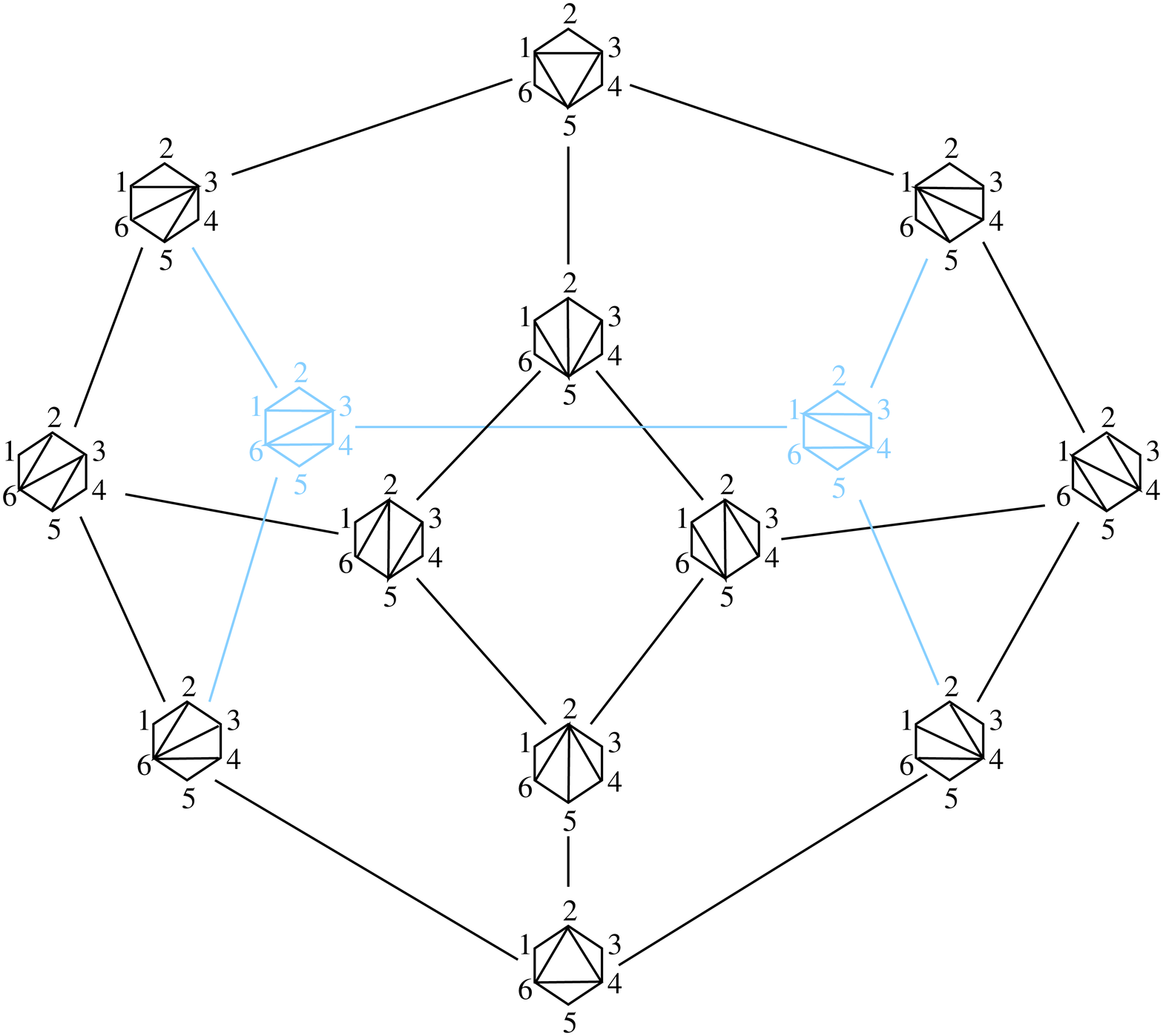}
\caption{The exchange graph of the cluster algebra of type $A_3$, 
which coincides with the $1$-skeleton of the associahedron.}
\label{fig:associahedron}
\end{figure}


\begin{exercise}\label{ex:flip}
Let $T$ be a triangulation of a polygon, and let 
$T'$ be the new triangulation obtained from $T$
by flipping the diagonal $k$.
Then the quiver associated to $T'$ is the same as the quiver
obtained from $Q(T)$ by mutating at $k$:
$Q(T') = \mu_k(Q(T))$.
\end{exercise}

\begin{prop}\label{prop:ind}
If $T_1$ and $T_2$ are two triangulations of a $d$-gon,
then the cluster algebras
$\Acal(T_1)$ and $\Acal(T_2)$ are isomorphic.
\end{prop}
\begin{proof}
This follows from 
Exercise \ref{ex:flip} and the fact that the flip-graph is connected.
\end{proof}

Since the cluster algebra associated to a triangulation of a $d$-gon
depends only on $d$, we will refer to this cluster algebra 
as $\Acal_{d-3}$.  We've chosen to index this cluster algebra by 
$d-3$ because this cluster algebra has rank $d-3$.

\subsubsection{The homogeneous coordinate ring of the Grassmannian $Gr_{2,d}$}

We now explain how the cluster algebra associated to a $d$-gon can be identified with the 
coordinate ring $\C[Gr_{2,d}]$ of (the affine cone over)
the Grassmannian $Gr_{2,d}$ of $2$-planes in a $d$-dimensional vector space. 

Recall that the coordinate ring
$\C[Gr_{2,d}]$ is generated by \emph{Pl\"ucker coordinates} 
$p_{ij}$ for $1 \leq i < j \leq d$.  The relations 
among the Pl\"ucker coordinates are generated by the \emph{three-term
Pl\"ucker relations}: for any $1 \leq i < j < k < \ell \leq d$,
one has  
\begin{equation}
\label{3-term}
p_{ik} p_{j\ell} = p_{ij} p_{k\ell} + p_{i\ell} p_{jk}.
\end{equation}

To make the connection with cluster algebras,
label the vertices of a $d$-gon from $1$ to $d$
in order around the boundary.  Then each side and diagonal of 
the polygon is uniquely identified by the labels of its endpoints.
Note that the Pl\"ucker coordinates for $Gr_{2,d}$  are in bijection with 
the set of sides and diagonals of the polygon, see the right of Figure
\ref{fig:octagon}.

By noting that  
the three-term Pl\"ucker relations correspond to  exchange relations
in $\Acal_{d-3}$, one may verify the following.

\begin{exercise}\label{ex:Ptolemy}
The cluster and coefficient variables of $\Acal_{d-3}$ are in bijection
with the diagonals and sides of the $d$-gon, and the clusters are in bijection
with triangulations of the $d$-gon.  Moreover
the coordinate ring of $Gr_{2,d}$ is isomorphic to the cluster algebra
$\Acal_{d-3}$ associated to the $d$-gon.
\end{exercise}

\begin{figure}[h]
\centering
\includegraphics[height=.7in]{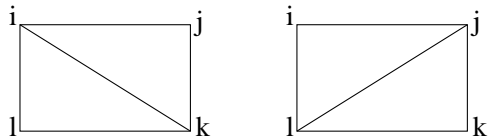}
\caption{A flip in a quadrilateral and the corresponding exchange relation $p_{ik} p_{j\ell} = p_{ij} p_{k\ell} + p_{i\ell} p_{jk}.$}
\label{fig:3-term}
\end{figure}

\begin{remark}
One may generalize this example of the type A cluster algebra
in several ways.  First, one may replace $Gr_{2,d}$ by 
an arbitrary Grassmannian, or partial flag variety.
It turns out that
the coordinate ring $\C[Gr_{k,d}]$ of any Grassmannian
has the structure of a cluster algebra \cite{Scott},
and more generally, so does the coordinate ring of any partial flag variety
$SL_m(\C)/P$ 
\cite{GLS-2008}.
Second, one may generalize 
this example by replacing the $d$-gon -- topologically a disk with 
$d$ marked points on the boundary -- by an orientable Riemann 
surface $S$ (with or without boundary) together with some marked points $M$
on $S$.
One may still consider triangulations  of $(S,M)$, and use the 
combinatorics of these triangulations to define a cluster algebra.
This cluster algebra is closely related to the 
\emph{decorated Teichmulller space} associated to $(S,M)$.
We will take up this theme in Section \ref{sec:Teichmuller}.
\end{remark}

\subsection{Cluster algebras revisited}\label{sec:generaldef}

We now give a more general definition of cluster algebra, 
following \cite{FZ4}, in which
the coefficient variables have their own dynamics.  In Section 
\ref{sec:Z} we will see that the dynamics of coefficient variables
is closely related to 
Zamolodchikov's Y-systems.

To define  a cluster algebra~$\Acal$ we first choose a 
\emph{semifield}
$(\PP,\oplus, \cdot)$, i.e.,
an abelian multiplicative group endowed with a binary operation of
\emph{(auxiliary) addition}~$\oplus$ which is commutative, associative, and
distributive with respect to the multiplication in~$\PP$.
The group ring~$\ZZ\PP$ will be
used as a \emph{ground ring} for~$\Acal$.
One important choice for $\PP$ is the tropical semifield 
(see Definition \ref{def:semifield-tropical}); in this case we say that the
corresponding cluster algebra is of {\it geometric type}.
\begin{definition}
[\emph{Tropical semifield}]
\label{def:semifield-tropical}
Let $\Trop (u_1, \dots, u_{m})$ be an abelian group, which is
freely generated by the $u_j$'s, and
written
multiplicatively.
We define  $\oplus$ in $\Trop (u_1,\dots, u_{m})$ by
\begin{equation}
\label{eq:tropical-addition}
\prod_j u_j^{a_j} \oplus \prod_j u_j^{b_j} =
\prod_j u_j^{\min (a_j, b_j)} \,,
\end{equation}
and call $(\Trop (u_1,\dots,u_{m}),\oplus,\cdot)$ a \emph{tropical
 semifield}.
Note that the group ring of $\Trop (u_1,\dots,u_{m})$ is the ring of Laurent
polynomials in the variables~$u_j\,$.
\end{definition}

As an \emph{ambient field} for
$\Acal$, we take a field $\Fcal$
isomorphic to the field of rational functions in $n$ independent
variables (here $n$ is the \emph{rank} of~$\Acal$),
with coefficients in~$\QQ \PP$.
Note that the definition of $\Fcal$ does not involve
the auxiliary addition
in~$\PP$.

\begin{definition}
[\emph{Labeled seeds}]
\label{def:seed}
A \emph{labeled seed} in~$\Fcal$ is
a triple $(\xx, \yy, B)$, where
\begin{itemize}
\item
$\xx = (x_1, \dots, x_n)$ is an $n$-tuple 
from $\Fcal$
forming a \emph{free generating set} over $\QQ \PP$, that is, 
$x_1, \dots, x_n$
are algebraically independent over~$\QQ \PP$, and
$\Fcal = \QQ \PP(x_1, \dots, x_n)$.
\item
$\yy = (y_1, \dots, y_n)$ is an $n$-tuple
from $\PP$, and
\item
$B = (b_{ij})$ is an $n\!\times\! n$ integer matrix
which is \emph{skew-symmetrizable}, that is, 
there exist positive integers $d_1,\dots,d_n$
such that $d_i b_{ij} = -d_j b_{ij}$.
\end{itemize}
We refer to~$\xx$ as the (labeled)
\emph{cluster} of a labeled seed $(\xx, \yy, B)$,
to the tuple~$\yy$ as the \emph{coefficient tuple}, and to the
matrix~$B$ as the \emph{exchange matrix}.
\end{definition}

We obtain ({\it unlabeled}) {\it seeds} from labeled seeds
by identifying labeled seeds that differ from
each other by simultaneous permutations of
the components in $\xx$ and~$\yy$, and of the rows and columns of~$B$.

In what follows, we
use the notation
$[x]_+ = \max(x,0)$.


\begin{definition}
[\emph{Seed mutations}]
\label{def:seed-mutation}
Let $(\xx, \yy, B)$ be a labeled seed in $\Fcal$,
and let $k \in \{1,\dots,n\}$.
The \emph{seed mutation} $\mu_k$ in direction~$k$ transforms
$(\xx, \yy, B)$ into the labeled seed
$\mu_k(\xx, \yy, B)=(\xx', \yy', B')$ defined as follows:
\begin{itemize}
\item
The entries of $B'=(b'_{ij})$ are given by
\begin{equation}
\label{eq:matrix-mutation}
b'_{ij} =
\begin{cases}
-b_{ij} & \text{if $i=k$ or $j=k$;} \\[.05in]
b_{ij}+b_{ik}b_{kj} & \text{if $b_{ik}>0$ and $b_{kj}>0$;}\\[.05in]
b_{ij}-b_{ik}b_{kj} & \text{if $b_{ik}<0$ and $b_{kj}<0$;}\\[.05in]
b_{ij} & \text{otherwise.}
\end{cases}
\end{equation}
\item
The coefficient tuple $\yy'=(y_1',\dots,y_n')$ is given by
\begin{equation}
\label{eq:y-mutation}
y'_j =
\begin{cases}
y_k^{-1} & \text{if $j = k$};\\[.05in]
y_j y_k^{[b_{kj}]_+}
(y_k \oplus 1)^{- b_{kj}} & \text{if $j \neq k$}.
\end{cases}
\end{equation}
\item
The cluster $\xx'=(x_1',\dots,x_n')$ is given by
$x_j'=x_j$ for $j\neq k$,
whereas $x'_k \in \Fcal$ is determined
by the \emph{exchange relation}
\begin{equation}
\label{exchange relation}
x'_k = \frac
{y_k \ \prod x_i^{[b_{ik}]_+}
+ \ \prod x_i^{[-b_{ik}]_+}}{(y_k \oplus 1) x_k} \, .
\end{equation}
\end{itemize}
\end{definition}

We say that two exchange matrices $B$ and $B'$ are {\it mutation-equivalent}
if one can get from $B$ to $B'$ by a sequence of mutations.

If we forget the cluster variables, then we refer to the resulting
seeds and operation of mutation as \emph{Y-seeds} and 
\emph{Y-seed mutation}.
\begin{definition}
[\emph{Y-seed mutations}]
\label{def:Yseed-mutation}
Let $(\yy, B)$ be a labeled seed in which we have omitted the 
cluster $\xx$,
and let $k \in \{1,\dots,n\}$.
The \emph{Y-seed mutation} $\mu_k$ in direction~$k$ transforms
$(\yy, B)$ into the \emph{labeled Y-seed}
$\mu_k(\yy, B)=(\yy', B')$, where 
$\yy'$ and $B'$ are  as in Definition \ref{def:seed-mutation}.
\end{definition}

\begin{definition}
[\emph{Patterns}]
\label{def:patterns}
Consider the \emph{$n$-regular tree}~$\TT_n$
whose edges are labeled by the numbers $1, \dots, n$,
so that the $n$ edges emanating from each vertex receive
different labels.
A \emph{cluster pattern}  is an assignment
of a labeled seed $\Sigma_t=(\xx_t, \yy_t, B_t)$
to every vertex $t \in \TT_n$, such that the seeds assigned to the
endpoints of any edge $t \overunder{k}{} t'$ are obtained from each
other by the seed mutation in direction~$k$.
The components of $\Sigma_t$ are written as:
\begin{equation}
\label{eq:seed-labeling}
\xx_t = (x_{1;t}\,,\dots,x_{n;t})\,,\quad
\yy_t = (y_{1;t}\,,\dots,y_{n;t})\,,\quad
B_t = (b^t_{ij})\,.
\end{equation}
\end{definition}

One may view a cluster pattern as a \emph{discrete dynamical system} on 
an $n$-regular tree.  
If we ignore the coefficients
(i.e. if we set each coefficient tuple equal to $(1,\dots,1)$ and choose
a semifield such that $1\oplus 1 = 1$), then 
we refer to the evolution of the cluster variables as
\emph{cluster dynamics}.  On the other hand, ignoring 
the cluster variables, we refer to the evolution 
of the coefficient variables  as \emph{coefficient dynamics}.

\begin{definition}
[\emph{Cluster algebra}]
\label{def:cluster-algebra}
Given a cluster pattern, we denote
\begin{equation}
\label{eq:cluster-variables}
\Xcal
= \bigcup_{t \in \TT_n} \xx_t
= \{ x_{i,t}\,:\, t \in \TT_n\,,\ 1\leq i\leq n \} \ ,
\end{equation}
the union of clusters of all the seeds in the pattern.
The elements $x_{i,t}\in \Xcal$ are called \emph{cluster variables}.
The 
\emph{cluster algebra} $\Acal$ associated with a
given pattern is the $\ZZ \PP$-subalgebra of the ambient field $\Fcal$
generated by all cluster variables: $\Acal = \ZZ \PP[\Xcal]$.
We denote $\Acal = \Acal(\xx, \yy, B)$, where
$(\xx,\yy,B)$
is any seed in the underlying cluster pattern.
\end{definition}

We now explain the relationship between 
Definition \ref{def:cluster-algebra0} -- 
the definition of cluster algebra
we gave in Section \ref{sec:first} --
and Definition \ref{def:cluster-algebra}.
There are two apparent differences between the definitions.
First, in Definition \ref{def:cluster-algebra0},
the dynamics of mutation was encoded by a quiver, while 
in Definition \ref{def:cluster-algebra}, the dynamics
of mutation was encoded by a skew-symmetrizable matrix $B$.
Clearly if $B$ is not only skew-symmetrizable but also skew-symmetric,
then $B$ can be regarded as the signed adjacency matrix
of a quiver.  In that case mutation of $B$
reduces to the mutation of the corresponding quiver, and the two
notions of exchange relation coincide.
Second, in Definition \ref{def:cluster-algebra0}, the 
coefficient variables are ``frozen" and do not mutate, 
while in Definition \ref{def:cluster-algebra}, the
coefficient variables $y_i$ 
have a dynamics of their own.  It turns out that if 
in Definition \ref{def:cluster-algebra} the semifield 
$\PP$ is the tropical semifield (and $B$ is skew-symmetric), then 
Definitions \ref{def:cluster-algebra0} and \ref{def:cluster-algebra}
are equivalent.  This is a consequence of the following exercise.

\begin{exercise}\label{ex:trop}
Let $\PP=\Trop(x_{n+1},\dots, x_m)$ be the tropical semifield
with generators $x_{n+1}, \dots, x_m$, and   
consider a cluster algebra as defined in Definition
\ref{def:cluster-algebra}.  Since the coefficients
$y_{j;t}$ at the seed $\Sigma_t = (\xx_t, \yy_t, B_t)$ 
are Laurent monomials in $x_{n+1},\dots,x_m$, we may define
the integers $b_{ij}^t$ for $j\in \{1,\dots,n\}$ and $n <i\leq m$ by
$$y_{j;t} = \prod_{i=n+1}^m x_i^{b_{ij}^t}.$$
This gives a natural way of including the exchange matrix $B_t$
as the principal $n \times n$ submatrix into a larger $m \times n$
matrix $\widetilde{B}_t = (b_{ij}^t)$ where $1\leq i \leq m$ and 
$1 \leq j \leq n$, whose matrix elements
$b_{ij}^t$ with $i>n$ encode the coefficients $y_j=y_{j;t}$.

Check that with the above conventions,
the exchange relation \eqref{exchange relation}
reduces to 
the exchange relation \eqref{exchange relation0}, and that the 
Y-seed mutation rule \eqref{eq:y-mutation} implies that 
the extended exchange matrix $\widetilde{B}_t$ mutates according to 
\eqref{eq:matrix-mutation0}.  
\end{exercise}

\subsection{Structural properties of cluster algebras}

In this section we will explain various structural properties of cluster
algebras.  
Throughout this section $\Acal$ will be an arbitrary cluster algebra as
defined in Section \ref{sec:generaldef}.

From the definitions, it is clear that any cluster variable
can be expressed as a rational function in the variables of 
an arbitrary cluster.  However, 
the remarkable {\it Laurent phenomenon}, proved in
\cite[Theorem 3.1]{FZ1},
 asserts that 
each such rational function is actually a Laurent polynomial.

\begin{theorem} 
[\emph{Laurent Phenomenon}]
\label{Laurent}
The cluster algebra $\Acal$ associated with a seed
$\Sigma = (\xx,\yy,B)$ is contained in the Laurent polynomial ring
$\ZZ \PP [\xx^{\pm 1}]$, i.e.\ every element of $\Acal$ is a
Laurent polynomial over $\ZZ \PP$ in the cluster variables
from $\xx=(x_1,\dots,x_n)$.
\end{theorem}

Let $\Acal$ be a cluster algebra, 
$\Sigma$ be a seed, and $x$
be a cluster variable of $\Acal$.  Let 
$[x]_{\Sigma}^{\Acal}$ denote the Laurent polynomial 
which expresses $x$ in 
terms of the cluster variables from $\Sigma$; it is called
the {\it cluster expansion} of $x$ in terms of $\Sigma$.
The longstanding {\it Positivity Conjecture} \cite{FZ1} says that 
the coefficients that appear in such Laurent polynomials
are positive.

\begin{conjecture} 
[\emph{Positivity Conjecture}]
\label{conj:pos}
For any cluster algebra $\Acal$, any seed $\Sigma$, and any cluster
variable $x$, the Laurent polynomial
$[x]_{\Sigma}^{\Acal}$ has coefficients which are 
nonnegative 
integer linear combinations of elements in $\PP$.
\end{conjecture}

While Conjecture \ref{conj:pos} is open in general, it has
been proved in some special cases, see for example 
\cite{CalRein}, \cite{MSW}, \cite{Nak}, \cite{KdF},
\cite{HL-Duke}, \cite{KimuraQin}.

One of Fomin-Zelevinsky's motivations for introducing cluster
algebras was the desire to understand the canonical bases
of quantum groups due to Lusztig and Kashiwara \cite{LuB, KaB}.
See  \cite{GLS-survey} for some recent results connecting
cluster algebras and canonical bases.
Some of the conjectures below, 
including Conjectures \ref{conj:ind} and \ref{conj:strong}, are
motivated in part by the conjectural connection between cluster algebras and 
canonical bases.

\begin{definition}
[\emph{Cluster monomial}]
A \emph{cluster monomial} in a cluster algebra $\Acal$ is a monomial 
in cluster variables, all of which belong to the same cluster.
\end{definition}

\begin{conjecture}\label{conj:ind}
Cluster monomials are linearly independent.
\end{conjecture}

The best result to date towards Conjecture \ref{conj:ind} is the following.
\begin{theorem}\cite{CKLP}
In a cluster algebra defined by a quiver, the cluster monomials
are linearly independent.
\end{theorem}

The following conjecture has long been a part of the cluster
algebra folklore, and it implies both Conjecture \ref{conj:pos}
and Conjecture \ref{conj:ind}.

\begin{conjecture} 
[\emph{Strong Positivity Conjecture}]
\label{conj:strong}
Any cluster algebra has an additive basis $\BB$ which
\begin{itemize}
\item includes the cluster monomials, and 
\item has nonnegative structure constants, that is, when one
writes the product of any two elements in $\BB$ in terms of $\BB$,
the coefficients are positive.
\end{itemize}
\end{conjecture}


One of the most striking results about cluster algebras is that 
the classification of the \emph{finite type} cluster algebras is parallel
to the Cartan-Killing classification of complex simple Lie algebras.
In particular, finite type cluster algebras are classified by 
Dynkin diagrams.

\begin{definition}
[\emph{Finite type}]
We say that a cluster algebra is of {\it finite type} if it has finitely many seeds.
\end{definition}

It turns out that the classification of finite type cluster algebras
is parallel to the Cartan-Killing classification of complex simple Lie algebras
\cite{FZ2}.
More specifically, define the \emph{diagram} $\Gamma(B)$
associated to an $n \times n$ exchange matrix $B$ to
be a weighted directed graph on 
nodes $v_1,\dots,v_n$,
with $v_i$ 
directed towards $v_j$ if and only if
$b_{ij}>0$.  In that case, we label this edge  by
$| b_{ij} b_{ji}|.$ 

\begin{theorem} \cite[Theorem 1.8]{FZ2}
\label{th:finite}
The cluster algebra $\Acal$ 
is of finite type if and only if
it has a seed $(\xx, \yy, B)$ such that
$\Gamma(B)$ is  an orientation of a finite type Dynkin diagram.
\end{theorem}

If the conditions of Theorem \ref{th:finite} hold, we say that 
$\Acal$ is of {\it finite type.}  And in that case 
if $\Gamma(B)$ is an orientation 
of a Dynkin diagram of type
$X$ (here $X$ belongs to one of the infinite series $A_n$, $B_n$, $C_n$, $D_n$, or to one of the 
exceptional types $E_6$, $E_7$, $E_8$, $F_4$, $G_2$),  we say that the 
cluster algebra $\Acal$ is of type $X$.

We define the \emph{exchange graph} of a cluster algebra to be the graph whose vertices
are the (unlabeled) seeds, and whose edges connect pairs of seeds which are connected by 
a mutation.
When a cluster algebra is of finite type, its exchange graph has a remarkable
combinatorial structure.

\begin{theorem}\cite{CFZ}
\label{th:assoc}
Let $\Acal$ be a cluster algebra of finite type.  Then its exchange graph is the 
$1$-skeleton of a convex polytope.
\end{theorem}

When $\Acal$ is a cluster algebra of type $A$, its exchange graph is the $1$-skeleton of 
a famous polytope called the 
\emph{associahedron}, see Figure \ref{fig:associahedron}.  Therefore the polytopes
that arise from finite type cluster algebras as in Theorem \ref{th:assoc} are called
\emph{generalized associahedra}.

\section{Cluster algebras in Teichm\"uller theory}
\label{sec:Teichmuller}

In this section we will explain how cluster algebras 
had already appeared implicitly in Teichm\"uller theory,  before the 
introduction of cluster algebras themselves.  
In particular, we will associate a cluster
algebra to any \emph{bordered surface with marked points}, following
work of Fock-Goncharov \cite{FG1}, Gekhtman-Shapiro-Vainshtein \cite{GSV},
and Fomin-Shapiro-Thurston \cite{FST}.
This construction provides a natural generalization of the type A 
cluster algebra  from
Section \ref{sec:A}, and 
realizes the \emph{lambda lengths} (also called \emph{Penner coordinates}) 
on the decorated 
Teichm\"uller space associated to a cusped surface,
which Penner had defined in 1987 \cite{Penner}.
We will also 
briefly discuss the Teichm\"uller space of a surface with 
oriented geodesic boundary
and related spaces of laminations, and how these spaces are related to 
cluster theory. 
For more details on the Teichm\"uller and lamination spaces, see \cite{FG3}.

\subsection{Surfaces, arcs, and triangulations}

\begin{definition}
[\emph{Bordered surface with marked points}]
\label{def:bordered}
Let $S$ be a connected oriented 2-dimensional Riemann surface with
(possibly empty)
boundary.  Fix a nonempty set $M$ of {\it marked points} in the closure of
$S$ with at least one marked point on each boundary component. The
pair $(S,M)$ is called a \emph{bordered surface with marked points}. Marked
points in the interior of $S$ are called \emph{punctures}.  
\end{definition}
 
For technical reasons we require that $(S,M)$ is not
a sphere with one, two or three punctures;
a monogon with zero or one puncture; 
or a bigon or triangle without punctures.


\begin{definition}
[\emph{Arcs and boundary segments}]
An \emph{arc} $\zg$ in $(S,M)$ is a curve in $S$, considered up
to isotopy, such that: 
the endpoints of $\zg$ are in $M$;
$\zg$ does not cross itself, except that its endpoints may coincide;
except for the endpoints, $\zg$ is disjoint from $M$ and
  from the boundary of $S$; and
$\zg$ does not cut out an unpunctured monogon or an unpunctured bigon. 

A \emph{boundary segment} is a curve that connects two
marked points and lies entirely on the boundary of $S$ without passing
through a third marked point.

\end{definition}
Let $A(S,M)$ and $B(S,M)$ denote the sets of arcs and boundary 
segments in $(S,M)$.  Note that $A(S,M)$ and 
$B(S,M)$ are disjoint.

\begin{definition}
[\emph{Compatibility of arcs, and triangulations}]
We say that arcs $\zg$ and $\zg'$ are \emph{compatible}
if there exist curves
$\za$ and $\za'$ isotopic to $\zg$ and $\zg'$,  
such that $\za$ and $\za'$ do not cross.
A \emph{triangulation} is a maximal collection of
pairwise compatible arcs (together with all boundary segments). 
The arcs of a 
triangulation cut the surface into \emph{triangles}. 
\end{definition}

There are two types of triangles: 
triangles that have three distinct sides, and \emph{self-folded triangles}
that have only two.   
Note that a self-folded triangle
consists of a loop, together with an arc to an enclosed puncture, 
called a \emph{radius}, 
see Figure \ref{fig:selffolded}.
\begin{figure}[h]
\centering
\includegraphics[height=1.2in]{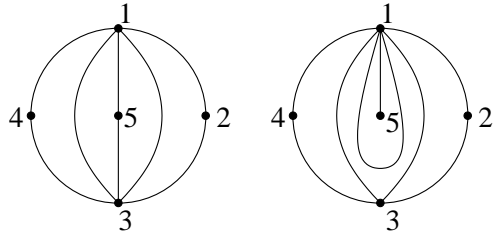}
\caption{Two triangulations of a once-punctured polygon.  The triangulation
at the right contains a self-folded triangle.}
\label{fig:selffolded}
\end{figure}

\begin{definition}
[\emph{Flips}]
A \emph{flip} of a triangulation $T$ replaces a single arc $\gamma\in T$ 
by a (unique) arc $\gamma' \neq \gamma$
that, together with the remaining arcs in $T$, forms a new 
triangulation.
\end{definition}

In Figure \ref{fig:selffolded}, the triangulation at the right
is obtained from the triangulation at the left by flipping the 
arc between the marked points $3$ and $5$.  However, 
a radius inside a self-folded triangle
in $T$ cannot be flipped (see e.g. the arc between $1$ and $5$ at the 
right).

\begin{prop}\cite{Harer, Hatcher, Mosher}
Any two triangulations of a bordered surface are related by a sequence of flips.
\end{prop}

\subsection{Decorated Teichm\"uller  space}
In this section we assume that the reader is familiar with 
some basics of hyperbolic geometry.
\begin{definition}
[\emph{Teichm\"uller space}]
Let $(S,M)$ be a bordered surface with marked points.  The 
(cusped) 
\emph{Teichm\"uller space} $\T(S,M)$ 
consists of all complete finite-area hyperbolic metrics
with constant curvature $-1$ on 
$S\setminus M$, with geodesic boundary at 
$\del S \setminus M$, considered up to 
$\Diff_0(S,M)$, diffeomorphisms of $S$ fixing $M$ that are homotopic to the 
identity.
(Thus there is a cusp at each point of $M$: points at $M$ ``go off 
to infinity," while the area remains bounded.)
\end{definition}


For a given hyperbolic metric in 
$\T(S,M)$, each arc can be represented by a unique geodesic.  Since there
are cusps at the marked points, such a geodesic segment is of infinite 
length.  So if we want to measure the ``length" of a geodesic arc between
two marked points, we need to renormalize.  

To do so, around each cusp $p$ we choose a \emph{horocycle},
which may be viewed as the set of points at an equal distance from $p$.
Although the cusp is infinitely
far away from any point in the surface, there is still a well-defined
way to compare the distance to $p$ from two different points in the surface.
A horocycle can also be characterized as a curve perpendicular to 
every geodesic to $p$.
See Figure \ref{fig:Ptolemy} for a depiction of some points and 
horocycles, drawn in the hyperbolic plane.

The notion of horocycle leads to the following definition.
\begin{definition}
[\emph{Decorated Teichm\"uller space}] A point in a decorated 
Teichm\"uller space $\TTT(S,M)$ is a hyperbolic metric as above
together with a collection of horocyles $h_p$, one around each cusp corresponding to a marked point $p\in M$.
\end{definition}

One may parameterize decorated Teichm\"uller space using 
\emph{lambda lengths} or \emph{Penner coordinates}, 
as introduced and developed by Penner 
\cite{Penner, Penner2}.

\begin{definition}
[\emph{Lambda lengths}] 
\cite{Penner2}
Fix $\sigma \in \TTT(S,M)$.  Let
$\gamma$ be an arc or a boundary segment.  Let $\gamma_{\sigma}$
denote the geodesic representative of $\gamma$ (relative to $\sigma$).
Let $\ell(\gamma) = \ell_{\sigma}(\gamma)$ be the signed distance
along $\gamma_{\sigma}$ between the horocycles at either end of 
$\gamma$ (positive if the two horocycles do not intersect, and negative
if they do).  The \emph{lambda length}
$\lambda(\gamma) = \lambda_{\sigma}(\gamma)$ of $\gamma$ is defined by
\begin{equation}
\lambda(\gamma) = \exp(\ell(\gamma)/2).
\end{equation}
\end{definition}

Given $\gamma \in A(S,M) \cup B(S,M)$, one may view the lambda length
\begin{equation*}
\lambda(\gamma): \sigma \mapsto \lambda_{\sigma}(\gamma)
\end{equation*}
as a function on the decorated Teichm\"uller space
$\TTT(S,M)$.  Let $n$ denote the number of arcs in a
triangulation of $(S,M)$; recall that $c$ denotes the number of marked points
on the boundary of $S$.  Penner showed that
if one fixes a triangulation $T$, then 
the lambda lengths of the  arcs of $T$ and the boundary segments
can be used to parameterize 
$\TTT(S,M)$:
\begin{theorem}\label{decTeich}
For any triangulation $T$ of $(S,M)$, the map
\begin{equation*}
\prod_{\gamma \in T \cup B(S,M)} \lambda(\gamma): \TTT(S,M) \to \R_{>0}^{n+c}
\end{equation*}
is a homeomorphism.
\end{theorem}
Note that the first versions of Theorem \ref{decTeich} were due to Penner
\cite[Theorem 3.1]{Penner}, \cite[Theorem 5.10]{Penner2}, but 
the formulation above is from 
\cite[Theorem 7.4]{FT}.

The following ``Ptolemy relation" is an indication that lambda lengths
on decorated Teichm\"uller space are part of a related cluster algebra.

\begin{prop}\cite[Proposition 2.6(a)]{Penner} \label{prop:Pt}
Let $\alpha, \beta, \gamma, \delta \in A(S,M) \cup B(S,M)$
be arcs or boundary segments (not necessarily distinct) that 
cut out a quadrilateral in $S$; we assume that the sides of the quadrilateral,
listed in cyclic order, are 
$\alpha, \beta, \gamma, \delta$.  Let $\eta$ and $\theta$ be the
two diagonals of this quadrilateral.  Then the corresponding 
lambda lengths satisfy the Ptolemy relation
\begin{equation*}
\lambda(\eta) \lambda(\theta) = \lambda(\alpha) \lambda(\gamma) + 
\lambda(\beta) \lambda(\delta).
\end{equation*}
\end{prop}

\begin{figure}[h]
\centering
\includegraphics[height=1.3in]{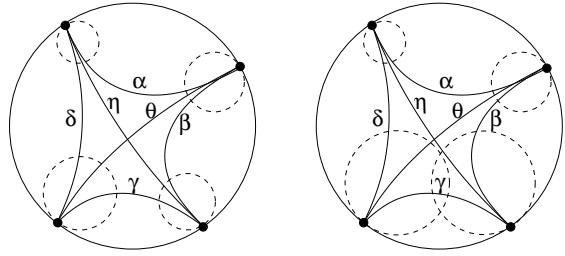}
\caption{Four points and corresponding horocycles, together with 
the arcs forming the geodesics between them, drawn in the hyperbolic plane.  
At the left, all lengths
$\ell(\alpha), \dots, \ell(\theta)$ are  positive; but at the right,
$\ell(\gamma)$ is negative.}
\label{fig:Ptolemy}
\end{figure}

\subsection{The cluster algebra associated to a surface}

To make precise the connection between decorated Teichm\"uller space
and cluster algebras, let us fix a triangulation $T$ of $(S,M)$,
and explain how to 
associate an exchange
matrix $B_T$ to $T$ \cite{FG1, FG2, GSV, FST}.  
For simplicity we assume that $T$ has no self-folded triangles.  It is not 
hard to see that all of the bordered surfaces we are considering admit
a triangulation without self-folded triangles.

\begin{definition}
[\emph{Exchange matrices associated to a triangulation}]
\label{def:adj}
Let $T$ be a triangulation of $(S,M)$.
Let $\tau_1,\tau_2,\ldots,\tau_n$ be 
the $n$ arcs of
$T$, and $\tau_{n+1},\dots,\tau_{n+c}$ be the $c$ boundary segments
of $(S,M)$.
We define 
{\small
\begin{align*}b_{ij} = 
&\#\{\text{triangles with sides } \tau_i \text{ and }\tau_j, 
\text{ with } \tau_j \text{ following } \tau_i
\text{ in clockwise order}\} -\\
&\#\{\text{triangles with sides } \tau_i \text{ and }\tau_j, 
\text{ with } \tau_j \text{ following } \tau_i
\text{ in counterclockwise order}\}.
\end{align*}}
Then we define the \emph{exchange matrix} 
$B_{T}=(b_{ij})_{1\le i\le n, 1\le j\le n}$ and the 
\emph{extended 
exchange matrix} 
$\widetilde{B}_{T}=(b_{ij})_{1\le i\le n+c, 1\le j\le n}$. 
\end{definition}

In Figure \ref{fig:Ptolemy} there is a triangle with sides 
$\alpha$, $\beta$, and $\eta$, and our convention is that 
$\beta$ follows $\alpha$ in clockwise order.  
Note that in order to speak about clockwise 
order, one must be working with an oriented surface.  That is why 
Definition \ref{def:bordered}
requires $S$ to be oriented.

We leave the following as an exercise for the reader;
alternatively, see \cite{FST}.

\begin{exercise}
Extend Definition \ref{def:adj} to the case
that $T$ has self-folded triangles, so that the exchange
matrices transform compatibly with mutation.
\end{exercise}


\begin{remark}\label{rem:2}
Note that each entry  $b_{ij}$ of the exchange matrix (or extended
exchanged matrix) is either
$0,\pm 1$, or $\pm 2$, since every arc $\tau$ is in at most two triangles. 
\end{remark}

\begin{exercise}
Note that $B_{T}$ is skew-symmetric, and hence can be viewed as the 
signed adjacency matrix associated to a quiver.
Verify that this quiver generalizes 
the quiver $Q(T)$ from
Definition \ref{def:QT} associated to a triangulation of a polygon.
\end{exercise}

The following result is proved by interpreting cluster variables as lambda lengths
of arcs, and using the fact that lambda lengths satisfy Ptolemy relations.
It also uses the fact that any two triangulations of $(S,M)$ can be connected
by a sequence of flips.

\begin{theorem}\label{th:surfacecombinatorics}
Let $(S,M)$ be a bordered surface and let $T=(\tau_1,\dots,\tau_n)$ 
be a triangulation of $(S,M)$.
Let $\xx_T=(x_{\tau_1}, \dots, x_{\tau_n})$, and  
let $\Acal = \Acal(\xx_T, B_T)$ be the corresponding cluster algebra.
Then we have the following:
\begin{itemize}
\item Each arc $\gamma \in A(S,M)$ gives rise 
to a cluster variable $x_{\gamma}$;
\item Each triangulation $T$ of $(S,M)$ gives rise to a seed 
$\Sigma_T = (x_T, B_T)$ of $\Acal$;
\item If $T'$ is obtained from $T$ by flipping at $\tau_k$, then
$B_{T'} = \mu_k(B_T)$.
\end{itemize}
It follows that the cluster algebra $\Acal$ does not depend on the 
triangulation $T$, but only on $(S,M)$.  Therefore we refer to 
this cluster algebra as $\Acal(S,M)$.
\end{theorem}

\begin{remark}\label{rem:tagged}
Theorem \ref{th:surfacecombinatorics} gives an inclusion of arcs into 
the set of cluster variables of $\Acal(S,M)$.  This inclusion is a 
bijection if and only if $(S,M)$ has no punctures.  
In \cite{FST}, Fomin-Shapiro-Thurston introduced 
\emph{tagged arcs} and \emph{tagged triangulations}, 
which generalize arcs and triangulations, and are in bijection with 
cluster variables and clusters of $\Acal(S,M)$, see
\cite[Theorem 7.11]{FST}.
To each tagged triangulation one may associate an 
exchange matrix, which as before has all entries equal to 
$0$, $\pm 1$, or $\pm 2$.
\end{remark}

Combining Theorem \ref{th:surfacecombinatorics} with 
Theorem \ref{decTeich} and Proposition \ref{prop:Pt},
we may identify the cluster variable $x_{\gamma}$ 
with the 
corresponding lambda length $\lambda(\gamma)$, 
and therefore view
such elements of the cluster algebra
$\Acal$ as functions on $\TTT(S,M)$.
In particular, when one performs a flip in a triangulation,
the lambda lengths associated to the arcs \emph{transform according to 
cluster dynamics}.

It is natural to consider whether there is a nice system of 
coordinates on Teichm\"uller space $\T(S,M)$ itself (as opposed to its 
decorated version.) Indeed, 
if one fixes a point of $\T(S,M)$ and a triangulation 
$T$ of $(S,M)$, 
one may represent $T$ by geodesics and lift it
to an \emph{ideal triangulation} of the upper half plane.
Note that every arc of the triangulation is the diagonal of 
a unique quadrilateral.  The four points of this quadrilateral
have a unique invariant under the action of 
$PSL_2(\R)$, the \emph{cross-ratio}. One may compute the  cross-ratio
by sending three of the four points to $0$, $-1$, and 
$\infty$.  Then the position $x$ of the fourth point is the cross-ratio.
The collection of cross-ratios associated to the arcs in $T$
comprise a system of coordinates on $\T(S,M)$.
And when one performs a flip in the triangulation, the 
coordinates \emph{transform according to coefficient dynamics},
see \cite[Section 4.1]{FG3}.

\subsection{Spaces of laminations and their coordinates}

Several compactifications of Teichm\"uller space have been introduced.  The 
most widely used compactification is due to W. Thurston 
\cite{Thurston, Thurston3}; 
the points at infinity of this compactification correspond to 
\emph{projective measured laminations.}

Informally, a 
measured lamination on $(S,M)$ is a finite
collection of non-self-intersecting and pairwise
non-intersecting weighted curves in $S \setminus M$, considered up to 
homotopy, and modulo a certain equivalence relation.
It is not hard to see why such a lamination $L$ 
might correspond to a limit point of Teichm\"uller space $\T(S,M)$:
given $L$, one may construct
a family of metrics on the surface ``converging to $L$,"
by cutting at each curve in $L$
and inserting a ``long neck."  As the necks get longer and longer,
the length of an arbitrary curve 
in the corresponding metric becomes 
dominated by the number of 
times that curve crosses $L$.  Therefore $L$ represents the limit 
of this family of metrics.

Interestingly, two versions of the  space of laminations on $(S,M)$ --
the space of \emph{rational bounded measured laminations},
and the space of \emph{rational unbounded measured laminations} --
are closely
connected to cluster theory.  
In both cases, one may fix a triangulation $T$ and then use
appropriate coordinates to get a parameterization of the space.
The appropriate coordinates for the space of rational bounded measured
laminations are \emph{intersection numbers}.  
Let $T'$ be a triangulation
obtained from $T$ by performing a flip.  It turns out that 
when one replaces $T$ by $T'$,
the rule for how intersection numbers change is given by 
a \emph{tropical} version of \emph{cluster dynamics}.
On the other hand, the appropriate coordinates for the space
of rational unbounded measured laminations are 
\emph{shear coordinates}.  When one replaces $T$ by $T'$,
the rule for how shear  coordinates change is given by a 
\emph{tropical} version of \emph{coefficient dynamics}.
See \cite{FG3} for more details.

In Table \ref{table:coordinates}, we summarize
the properties of the two versions of Teichm\"uller space,
and the two versions of the space of laminations, together
with their coordinates.
\begin{table}[h]
\small
\begin{tabular}{|l|l|l|}
\hline
Space  & Coordinates & Coordinate transformations \\
\hline
\hline
Decorated Teichm\"uller space  & Lambda lengths & Cluster dynamics \\
\hline
Teichm\"uller space & Cross-ratios & Coefficient dynamics\\
\hline
Bounded measured laminations & Intersection numbers & Tropical cluster dynamics\\
\hline
Unbounded measured laminations & Shear coordinates & Tropical coefficient dynamics\\
\hline
\end{tabular}
\bigskip
\caption{Teichm\"uller and lamination spaces.}
\label{table:coordinates}
\end{table}

\subsection{Applications of Teichm\"uller theory to cluster theory}

The connection between Teichm\"uller theory and cluster algebras
has useful applications to cluster algebras, some of which we discuss
below.

As mentioned in Remark \ref{rem:tagged},
the combinatorics of (tagged) arcs and (tagged)
triangulations gives a concrete way
to index cluster variables and clusters in a cluster algebra
from a surface.
Additionally, the combinatorics of unbounded measured 
laminations gives a concrete way 
to encode the coefficient variables for a cluster algebra, whose coefficient
system is of
geometric type \cite{FT}.  Recall from Section \ref{sec:first}
or Exercise \ref{ex:trop} that the coefficient system
is determined by the bottom $m-n$ rows of the initial extended exchange matrix
$\widetilde{B}$.  However, after one has mutated away from the initial
cluster, one would like an explicit way to read off the resulting 
coefficient variables (short of performing the corresponding 
sequence of mutations).
In \cite{FT}, the authors demonstrated that one may encode the initial
extended exchange matrix by a triangulation \emph{together with a 
lamination}, and that one may compute the coefficient variables
(even after mutating away from the initial cluster)
by using \emph{shear coordinates}.

Note that for a general cluster algebra, 
there is no explicit way to index cluster variables or clusters, or to encode
the coefficients.  A cluster variable is simply a rational function
of the initial cluster variables that is obtained after some 
arbitrary and arbitrarily long sequence of mutations.  Having 
a concrete index set for the cluster variables and clusters,
as in \cite[Theorem 7.11]{FST}, is a powerful tool.  Indeed, this was
a key ingredient in \cite{MSW}, which proved the 
Positivity Conjecture for all cluster algebras from surfaces.

The connection between Teichm\"uller theory and cluster algebras
from surfaces
has also led to important structural results for such cluster algebras.
We say that a cluster algebra has \emph{polynomial growth} if the number
of distinct seeds which can be obtained from a fixed initial seed by
at most $n$ mutations is bounded from above by a polynomial
function of $n$. A cluster algebra has \emph{exponential growth}
if the number of such seeds is bounded from below by an exponentially
growing function of $n$. 
In \cite{FST}, 
Fomin-Shapiro-Thurston classified the cluster algebras from surfaces
according to their growth: there are six infinite families which have
polynomial growth, and all others have exponential growth.

Another structural result relates to the classification of 
mutation-finite cluster algebras.
We say that a matrix $B$ (and the corresponding cluster algebra) 
is \emph{mutation-finite} (or is of 
\emph{finite mutation type}) if its mutation
equivalence class is finite, i.e. only finitely many matrices can be obtained
from $B$ by repeated matrix mutations.
Felikson-Shapiro-Tumarkin gave a 
classification of all skew-symmetric mutation-finite cluster algebras
in \cite{FeSTu}.
They showed that these cluster algebras are the union of the following classes
of cluster algebras:
\begin{itemize}
\item Rank 2 cluster algebras;
\item Cluster algebras from surfaces;
\item One of 11 exceptional types.
\end{itemize}
Note that the above classification may be extended to
all mutation-finite cluster algebras (not necessarily skew-symmetric),
using \emph{cluster algebras from orbifolds} 
\cite{FeSTu2}.

\begin{exercise}
Show that any cluster algebra from a surface is 
mutation-finite.  \emph{Hint: use Remark \ref{rem:tagged}}.
\end{exercise}


\section{Cluster algebras and the Zamolodchikov periodicity conjecture}\label{sec:Z}

The thermodynamic Bethe ansatz  is a tool 
for understanding certain conformal field theories.
In a paper from 1991 \cite{Z}, the physicist 
Al.\ B. Zamolodchikov studied the thermodynamic
Bethe ansatz equations for ADE-related diagonal scattering theories.
He showed that if one has a solution to these
equations, it should also be a solution of a set of 
functional relations called a \emph{Y-system}.
Furthermore, he remarked that based on numerical tests, 
the solutions to the Y-system appeared to be periodic.
This phenomenon is called the \emph{Zamolodchikov periodicity
conjecture}, and has important consequences for 
the corresponding field theory.
Although this conjecture arose in mathematical physics,
we will see that it can be reformulated and proved using 
the framework of cluster algebras.


Note that Zamolodchikov initially stated his conjecture for 
the Y-system of a
 simply-laced Dynkin diagram.  The notion of Y-system and the periodicity
conjecture were subsequently 
generalized 
by Ravanini-Valleriani-Tateo \cite{RVT},
Kuniba-Nakanishi \cite{KN}, Kuniba-Nakanishi-Suzuki \cite{KNS},
Fomin-Zelevinsky \cite{FZY}, etc.
We will first present Zamolodchikov's periodicity conjecture for Dynkin diagrams
$\Delta$
(not necessarily simply-laced),
and then present its extension to pairs $(\Delta, \Delta')$ of Dynkin diagrams. 
Note that the latter conjecture reduces to the former in the case that 
$\Delta' = A_1$.
The conjecture was proved 
for $(A_n, A_1)$ by Frenkel-Szenes \cite{FS} and 
Gliozzi-Tateo \cite{GT}; 
for $(\Delta, A_1)$ (where $\Delta$ is an arbitrary Dynkin diagram) by
Fomin-Zelevinsky \cite{FZY}; and for 
$(A_n, A_m)$ by Volkov \cite{V} and independently by 
Szenes \cite{Sz}.  Finally in 2008, Keller 
proved the conjecture in the general case \cite{Keller1, Keller2}, 
using cluster algebras and their additive categorification via
triangulated categories.  Another proof 
was subsequently given by 
Inoue-Iyama-Keller-Kuniba-Nakanishi \cite{IIKKN1, IIKKN2}.

In Sections \ref{sec:Z1} and \ref{sec:Z2} we will state the periodicity
conjecture for Dynkin diagrams and pairs of Dynkin diagrams, respectively,
and explain how the conjectures may be formulated in terms of 
cluster algebras.  In Section \ref{sec:Z3} we will  discuss how 
techniques from the theory of cluster algebras were used to prove
the conjectures.


\subsection{Zamolodchikov's Periodicity Conjecture for Dynkin diagrams}
\label{sec:Z1}

Let $\Delta$ be a Dynkin diagram with vertex set
$I$.  Let $A$ denote the incidence matrix of $\Delta$, i.e. if $C$ is the
Cartan matrix of $\Delta$ and $J$ the identity matrix of the same size,
then $A = 2J-C$.  
Let $h$ denote the Coxeter number of $\Delta$, 
see Table \ref{table:h}.
\begin{table}[h]
\begin{tabular}{|c||c|c|c|c|c|c|c|c|c|}
\hline
$\Delta$ & $A_n$ & $B_n$ & $C_n$ & $D_n$ & $E_6$ & $E_7$ & $E_8$ & $F_4$ & $G_2$ \\
\hline
$h$ & $n+1$ & $2n$ & $2n$ & $2n-2$ & $12$ & $18$ & $30$ & $12$ & $6$\\
\hline
\end{tabular}
\bigskip
\caption{Coxeter numbers.}
\label{table:h}
\end{table}
\vspace{-.2cm}
\begin{theorem} 
[\emph{Zamolodchikov's periodicity conjecture}]
\label{th:Z1}
Consider the recurrence relation
\begin{equation}\label{eq:Z1}
Y_i(t+1) Y_i(t-1) = \prod_{j\in I} (Y_j(t)+1)^{a_{ij}},\quad t\in \Z.
\end{equation}
All solutions to this system are periodic in $t$ with period 
dividing $2(h+2)$, i.e. 
$Y_i(t+2(h+2)) = Y_i(t)$ for all $i$ and $t$.
\end{theorem}

The system of equations  in \eqref{eq:Z1} is called a 
\emph{Y-system.}


Note that any Dynkin diagram is a tree, and hence its set $I$ of vertices
is the disjoint union of two sets $I_+$ and $I_-$ such that there is no 
edge between any two vertices of $I_+$ nor between any two vertices of 
$I_-$.  Define $\epsilon(i)$ to be $1$ or $-1$ based on whether
$i \in I_+$ or $i\in I_-$.
Let $\Q(u)$ be the field of rational functions in the variables
$u_i$ for $i\in I$.  For $\epsilon = \pm 1$, define an 
automorphism $\tau_\epsilon$ by setting
\begin{equation*}
\tau_{\epsilon}(u_i) = 
     \begin{cases}
      u_i \prod_{j\in I} (u_j+1)^{a_{ij}}   &\text{ if  $\epsilon(i) = \epsilon$}\\
      u_i^{-1}                               &\text{ otherwise.}
\end{cases}
\end{equation*}

One may reformulate Zamolodchikov's periodicity conjecture in terms of 
$\tau_{\epsilon}$, as we will see below
in Lemma \ref{lem:Z1}.  
First note that the variables
$Y_i(k)$ on the left-hand side of 
\eqref{eq:Z1} have 
a fixed ``parity"  $\epsilon(i) (-1)^k$.  Therefore 
the Y-system decomposes into two independent systems, 
an even one and an odd one, and it suffices to prove
periodicity for one of them.  Without
loss of generality, we may therefore assume that 
\begin{equation*}
Y_i(k+1) = Y_i(k)^{-1} \text{ whenever } \epsilon(i) = (-1)^k.
\end{equation*}
If we combine this assumption with \eqref{eq:Z1}, we obtain
\begin{equation}
\label{newY} Y_i(k+1)= 
     \begin{cases}
     Y_i(k) {\prod_{j\in I} (Y_j(k)+1)^{a_{ij}}}    
                   &\text{ if  $\epsilon(i) = (-1)^{k+1}$}\\
      Y_i(k)^{-1}          &\text{ if $\epsilon(i)=(-1)^k$.} 
\end{cases}
\end{equation}

\begin{example}\label{ex:A2}
Let $\Delta$ be the Dynkin diagram of type $A_2$, on nodes $1$ and $2$,
where $I_-=\{1\}$ and $I_+ = \{2\}$.  The incidence matrix of the Dynkin
diagram is 
\[
A = \begin{pmatrix}
0 & 1 \\
1 & 0
\end{pmatrix}.
\]

If we set $Y_1(0)=u_1$, $Y_2(0)=u_2$,
then the recurrence for $Y_i(k)$ in \eqref{newY} yields:

\begin{align*}
Y_1(0)&=u_1 \,,& Y_2(0)&=u_2\\
Y_1(1)&=u_1^{-1} \,,& Y_2(1)&=u_2(1+u_1)\\
Y_1(2)&=\frac{1+u_2+u_1 u_2}{u_1} \,,&  Y_2(2)&=\frac{1}{u_2(1+u_1)}\\
Y_1(3)&=\frac{u_1}{1+u_2+u_1 u_2} \,,&  Y_2(3)&=\frac{1+u_2}{u_1 u_2}\\
Y_1(4)&=u_2^{-1} \,,&         Y_2(4)&=\frac{u_1 u_2}{1+u_2}\\
Y_1(5)&={u_2} \,,&                Y_2(5)&=u_1.
\end{align*}
By symmetry, it's clear that $Y_1(10)=u_1$ and $Y_2(10)=u_2$
and this system has period $10=2(3+2)$, as predicted by Theorem 
\ref{th:Z1}.
\end{example}

The following lemma follows easily from induction and the definition of 
$\tau_{\epsilon}$.
\begin{lemma}\label{lem:Z}
Set $Y_i(0)=u_i$ for $i\in I$.  Then for all 
$k \in \Z_{\geq 0}$ and $i \in I$, we have
$Y_i(k) = (\tau_- \tau_{+} \dots \tau_{\pm})(u_i)$, where
the number of factors $\tau_+$ and $\tau_-$ equals $k$.
\end{lemma}

Let us define an automorphism  of $\Q(u)$ by 
\begin{equation}
\phi = \tau_- \tau_+.
\end{equation}

Then we have the following.
\begin{lemma}\label{lem:Z1}
The Y-system from \eqref{eq:Z1} is periodic with period dividing
$2(h+2)$ if and only if $\phi$ has finite order dividing $h+2$.
\end{lemma}


To connect Zamolodchikov's conjecture to cluster algebras, let
us revisit the notion of Y-seed mutation from 
Definition \ref{def:Yseed-mutation}.  We will assume that $B$
is skew-symmetric, and hence can be encoded by a finite
quiver $Q$ without loops or $2$-cycles.
Let $(\PP,\oplus, \cdot)$ be $\Q$ with the usual operations of 
addition and multiplication.  Then $\mu_k(\yy, Q) = 
(\yy', Q')$, where 
\begin{equation*}
y_j' = 
     \begin{cases}
      y_k^{-1}   &\text{ if  $j=k$}\\
      y_j (1+y_k)^{m} &\text{ if there are $m\geq 0$ arrows $j \to k$}\\
      y_j (1+y_k^{-1})^{-m} &\text{ if there are $m \geq 0$ arrows $k \to j$.}
\end{cases}
\end{equation*}

Comparing the formula for Y-seed mutation with the definition
of the
automorphisms $\tau_{\epsilon}$ suggests a connection, 
which we make precise in Exercise \ref{ex:Z1}.
First we will define a \emph{restricted $Y$-pattern}.

\begin{definition}
[\emph{Restricted $Y$-pattern}]
Let $(\PP,\oplus, \cdot)$ be $\Q$ with the usual operations of 
addition and multiplication.  Let $Q$ denote a finite quiver
without loops or $2$-cycles with vertex set $\{1,\dots,n\}$, 
let $\yy=(y_1,\dots,y_n)$ and let
$(\yy,Q)$ be the corresponding Y-seed. 
Let $\vv$ be a sequence of vertices $v_1,\dots,v_N$ of $Q$,
with the property that the composed mutation
$$\mu_{\vv} = \mu_{v_N} \dots \mu_{v_2} \mu_{v_1}$$
transforms $Q$ into itself.  Then clearly the same holds
for the same sequence in reverse $\mu_{\vv}^{-1}$.  We 
define the \emph{restricted $Y$-pattern} associated with $Q$
and $\mu_{\vv}$ to be the sequence of Y-seeds obtained from 
the initial Y-seed $(\yy,Q)$ be applying all integer powers
of $\mu_{\vv}$.
\end{definition}


\begin{exercise}\label{ex:Z1}
Let $\Delta$ be a simply-laced Dynkin diagram with $n$ vertices, and vertex set
$I = I_+ \cup I_-$ as above.  
Let $Q$ denote  the unique ``bipartite" orientation 
of $\Delta$ such that each vertex in 
$I_+$ is a source and each vertex in $I_-$ is a sink.
\begin{enumerate}
\item Then the composed mutation $\mu_+ = \prod_{i\in I_-} \mu_i$ is well-defined, 
in other words, any sequence of mutations on the vertices in $I_-$
yields the same result.  Similarly $\mu_- = \prod_{i\in I_+} \mu_i$ 
is well-defined.
\item The composed mutation $\mu_-  \mu_+$ transforms $Q$ into itself.
Similarly for 
$\mu_+  \mu_- $.  
\item 
The automorphism $\tau_- \tau_+$ has finite order $m$ if and only if
the restricted $Y$-pattern associated with $Q$ and $\mu_- \mu_+$
is periodic with period $m$.
\end{enumerate}
\end{exercise}

Combining Exercise \ref{ex:Z1} with Lemma \ref{lem:Z}, we see that 
Zamolodchikov's periodicity conjecture is equivalent to 
verifying the periodicity of the restricted $Y$-pattern
from Exercise \ref{ex:Z1} (3).
See Figure \ref{fig:Yseeds} for an example of Y-seed mutation
on the Dynkin diagram of type $A_2$.  Compare the labeled Y-seeds here
with Example \ref{ex:A2}.
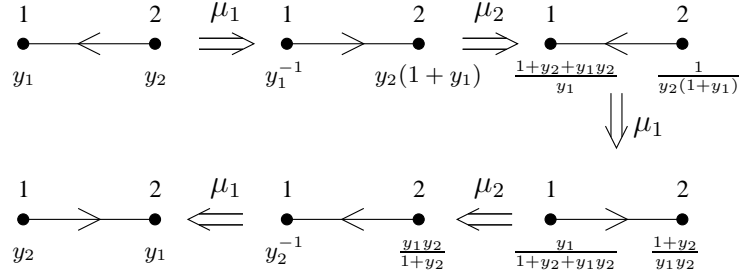
\begin{figure} \begin{center}
\input{Y-A2.pstex_t}
\end{center}
\caption{Y-seeds and Y-seed mutations in type $A_2$.}
\label{fig:Yseeds}
\end{figure}


\subsection{The periodicity  conjecture for pairs of Dynkin diagrams}\label{sec:pairs}
\label{sec:Z2}

In this section we let $\Delta$ and $\Delta'$ be Dynkin diagrams,
with vertex sets $I$ and $I'$, and incidence matrices $A$ and $A'$.

\begin{theorem} 
[\emph{The periodicity conjecture for pairs of Dynkin diagrams}]
\label{th:Z2}
Consider the recurrence relation
\begin{equation}\label{eq:Z2}
Y_{i,i'}(t+1) Y_{i,i'}(t-1) =\frac{ \prod_{j\in I} (Y_{j,i'}(t)+1)^{a_{ij}}} 
                                  {\prod_{j'\in I'} (Y_{i,j'}(t)^{-1}+1)^{a'_{i',j'}}},\quad t\in \Z.
\end{equation}
All solutions to this system are periodic in $t$ with period 
dividing $2(h+h')$.
\end{theorem}

Note that if $\Delta'$ is of type $A_1$, then 
Theorem \ref{th:Z2} reduces to Theorem \ref{th:Z1}.

Just as we saw for Theorem \ref{th:Z1}, it is possible to reformulate
Theorem \ref{th:Z2} in terms
of certain automorphisms.  Write
$I = I_+ \cup I_-$ 
and $I' = I'_+ \cup I'_-$
as before.
For a vertex 
$(i,i')$ of the product $I \times I'$, define
$\epsilon(i,i')$ to be $1$ or $-1$ based on whether 
$(i,i')$ lies in $(I_+ \times I'_+) \cup (I_- \times I'_-)$
or not.   Let $\Q(u)$ be the field of rational functions
in the variables $u_{ii'}$ for $i\in I$ and $i'\in I'$,
and define an automorphism of $\Q(u)$ by 
\begin{equation}
\label{eq:phi2}
\phi = \tau_- \tau_+, \text{ where }
\end{equation}
\begin{equation*}
\tau_{\epsilon}(u_{ii'}) = 
     \begin{cases}
      u_{ii'} 
         \prod_{j\in I} (u_{ji'}+1)^{a_{ij}} 
         \prod_{j'\in I'} (u_{ij'}^{-1}+1)^{-a'_{i'j'}} 
                  &\text{ if  $\epsilon(i,i') = \epsilon$}\\
      u_{ii'}^{-1}                               &\text{ otherwise.}
\end{cases}
\end{equation*}

As before, we may assume that 
$Y_{i,i'}(k+1) = Y_{i,i'}(k)^{-1}$  whenever 
$\epsilon(i,i')=(-1)^{k}.$
One may then reformulate the periodicity conjecture
as follows.

\begin{lemma}\label{lem:Z2}
The periodicity conjecture for pairs of Dynkin diagrams
holds if and only if 
$\phi$ has finite order dividing $h+h'$.
\end{lemma}

Now let us explain how to relate the periodicity conjecture
for pairs of Dynkin diagrams
to cluster algebras (more specifically, restricted $Y$-patterns).
We first need to define some operations on quivers.

Let $Q$ and $Q'$ be two finite quivers on vertex sets $I$ and $I'$
which are bipartite, i.e.
each vertex is a source or a sink.  The \emph{tensor product}
$Q \otimes Q'$ is the quiver on vertex set
$I \times I'$,  where the number of arrows from
a vertex $(i,i')$ to a vertex $(j,j')$ 
\begin{enumerate}
\item is zero if $i \neq j$ and $i' \neq j'$;
\item equals the number of arrows from $j$ to $j'$ if $i=i'$;
\item equals the number of arrows from $i$ to $i'$ if $j=j'$.
\end{enumerate}
The \emph{square product} $Q\square Q'$ is
the quiver obtained from $Q \otimes Q'$ by reversing all arrows
in the full subquivers of the form 
$\{i\} \times Q'$ and $Q \times \{i'\}$, where $i$ is a sink of $Q$
and $i'$ a source of $Q'$.  See Figure \ref{fig:a4-prod-d5} 
for an example of the 
square product of the following quivers.
\begin{align*}
\vec{A}_4 &: \xymatrix{ 1 & 2 \ar[l] \ar[r] & 3 & 4 \ar[l] } \ko\\
\vec{D}_5 &: \raisebox{0.75cm}{\xymatrix@R=0.2cm{   &  &  & 4 \ar[dl] \\ 1 & 2 \ar[l] \ar[r] & 3 & \\
& & & 5. \ar[ul] }}
\end{align*}

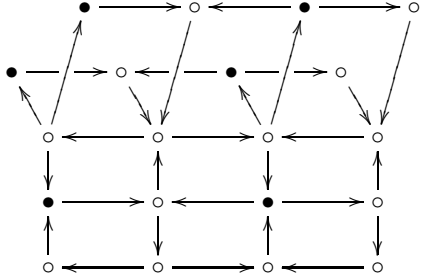
\begin{figure}
\[
\xymatrix@C=0.1cm@R=0.5cm{ & & \bt \ar[rrr] &  &  & \circ \ar[ldd]
& & &
\bt \ar[lll] \ar[rrr] & & & \circ \ar[ldd] \\
     \bt \ar[rrr]|!{[dr];[urr]}\hole & &  & \circ \ar[rd] & & &
\bt \ar[lll]|!{[lu];[dll]}\hole \ar[rrr]|!{[rd];[rru]}\hole & & & \circ \ar[rd] & &  \\
           & \circ \ar[lu] \ar[ruu] \ar[d] & & & \circ \ar[lll] \ar[rrr] & & &
\circ \ar[ruu] \ar[lu] \ar[d] & & & \circ \ar[lll] & \\
           & \bt \ar[rrr] & & & \circ \ar[d] \ar[u] & & &
\bt \ar[lll] \ar[rrr] & & & \circ \ar[u] \ar[d] & \\
           & \circ \ar[u] & & & \circ \ar[lll] \ar[rrr] &  & &
\circ \ar[u] & & & \circ \ar[lll] & }
\]
\caption{The quiver   $\vec{A}_4 \square \vec{D}_5$}
\label{fig:a4-prod-d5}
\end{figure}

\begin{exercise}\label{ex:Z2}
Let $\Delta$ and $\Delta'$ be simply-laced Dynkin diagrams with 
vertex sets $I$ and $I'$.
We write 
$I = I_+ \cup I_-$ 
and
$I' = I'_+ \cup I'_-$ as usual, and choose the corresponding
bipartite orientations $Q$ and $Q'$ of $\Delta$ and $\Delta'$
so that each vertex in $I_+$ or $I'_+$ is a source
and each vertex in $I_-$ or $I'_-$ is a sink.  
\begin{enumerate}
\item Given two elements $\sigma, \sigma'$ of $\{+,-\}$,  the
following composed mutation 
$$\mu_{\sigma, \sigma'} = \prod_{i \in I_{\sigma},\ i'\in I'_{\sigma'}}
\mu_{(i,i')}$$ of $Q \square Q'$ is well-defined, 
that is, the order in the product does not matter.
\item The composition
$\mu_{\square} = \mu_{-,-}\ \mu_{+,+}\ \mu_{-,+}\ \mu_{+,-}$
transforms $Q$ into itself.
\item 
The automorphism $\phi$ has finite order $m$ if and only if
the restricted Y-seed associated with $Q\square Q'$ and $\mu_{\square}$
is periodic with period $m$.  
\end{enumerate}
\end{exercise}

Combining Exercise \ref{ex:Z2} with Lemma \ref{lem:Z2}, we have the following:
\begin{lemma}
The periodicity conjecture holds for $\Delta$ and $\Delta'$ if and only if
the restricted Y-seed associated with 
$Q\square Q'$ and $\mu_{\square}$
is periodic with period dividing $h+h'$.
\end{lemma}

\subsection{On the proofs of the periodicity conjecture}
\label{sec:Z3}

In this section we discuss how 
techniques from the theory of cluster algebras were used to 
prove Theorems \ref{th:Z1} and \ref{th:Z2}.

First note that the proofs of Theorem \ref{th:Z1} and Theorem \ref{th:Z2} can 
be reduced to the case that the Dynkin diagrams are simply-laced,
using standard ``folding" arguments.  
Second,
as illustrated in Exercises \ref{ex:Z1} and \ref{ex:Z2},
the periodicity conjecture for simply-laced Dynkin diagrams
may be reformulated in terms of 
cluster algebras.  Specifically, the conjecture is equivalent
to verifying the periodicity of certain restricted $Y$-patterns.

Fomin-Zelevinsky's proof of Theorem \ref{th:Z1} 
used ideas which are now fundamental to the 
structure theory of finite type cluster algebras,
including a bijection between cluster variables and 
``almost-positive" roots of the corresponding root system.
They showed that this bijection, together with a  ``tropical" version of 
Theorem \ref{th:Z1}, implies Theorem \ref{th:Z1}.
Moreover, they gave an explicit solution to each Y-system,
in terms of certain \emph{Fibonacci polynomials}.
The Fibonacci polynomials
are (up to a twist) special cases of \emph{F-polynomials}, which 
in turn are important objects in cluster algebras, and control 
the dynamics of both cluster and coefficient variables
\cite{FZ4}.  
Note however that the Fomin-Zelevinsky proof does not apply 
to Theorem \ref{th:Z2}, because the cluster algebras 
associated with products $Q \square Q'$ are not in general of finite type.

Keller's proof of Theorem \ref{th:Z2} used the 
\emph{additive categorification} (via triangulated categories)
of cluster algebras.  To give some background on categorification,
in 2003, Marsh-Reineke-Zelevinsky \cite{MRZ} discovered that 
when $\Delta$ is a simply-laced Dynkin diagram, there is a 
close resemblance between the combinatorics of the cluster
variables and those of the \emph{tilting modules}
in the category of representations of the quiver; this initiated
the theory of \emph{additive categorification} of cluster algebras.
In this theory, one seeks to construct module or triangulated
categories associated to quivers so as to obtain
a correspondence between rigid objects of the categories
and the cluster monomials in the cluster algebras.  The required
correspondence sends direct sum decompositions
of rigid objects to factorizations of the associated cluster 
monomials.  
One may then hope to use the rich structure of these 
categories  to prove results on cluster algebras which 
seem beyond the scope of purely combinatorial methods.

Recall from Section \ref{sec:pairs} that 
the periodicity
conjecture for pairs of Dynkin diagrams is equivalent to 
the periodicity of the automorphism $\phi$ from \eqref{eq:phi2},
which in turn is equivalent to the periodicity 
of a restricted $Y$-pattern associated to $Q \square Q'$.
Keller's central construction from \cite{Keller2} was a 
triangulated 2-Calabi-Yau category $\CC$ with a cluster-tilting
object $T$, whose endoquiver (quiver of its endomorphism algebra)
is closely related to $Q \square Q'$; the category $\CC$ is a
generalized cluster category in the sense of Amiot \cite{A}.
Since $\CC$ is 2-Calabi-Yau, results of Iyama-Yoshino \cite{IY} imply
that there is a well-defined
mutation operation for the cluster-tilting objects.
Keller defined  the \emph{Zamolodchikov transformation}
$Za: \CC \to \CC$, which one may think of as 
a categorification of the automorphism 
$\phi$, and proved that $Za$
is periodic of period $h+h'$.  By ``decategorification," it follows
that $\phi$ is periodic of period $h+h'$, and hence the periodicity
conjecture for pairs of Dynkin diagrams is true.

The Inoue-Iyama-Keller-Kuniba-Nakanishi proof of Theorem \ref{th:Z2}
also used categorification, in particular the work of Plamondon 
\cite{Plamondon}.  Moreover, just as in the case of the Fomin-Zelevinsky
proof of Theorem \ref{th:Z2}, one crucial ingredient in their proof was
a ``tropical" version of Theorem \ref{th:Z2}.


\bibliographystyle{amsplain}
\bibliography{CA.bib}

\end{document}

%% file: A2.pstex_t
\begin{picture}(0,0)%
\includegraphics{A2.pstex}%
\end{picture}%
\setlength{\unitlength}{1816sp}%
\begingroup\makeatletter\ifx\SetFigFont\undefined%
\gdef\SetFigFont#1#2#3#4#5{%
  \reset@font\fontsize{#1}{#2pt}%
  \fontfamily{#3}\fontseries{#4}\fontshape{#5}%
  \selectfont}%
\fi\endgroup%
\begin{picture}(9274,3894)(1036,-3691)
\put(1051,-1186){\makebox(0,0)[lb]{\smash{{\SetFigFont{9}{10.8}{\rmdefault}{\mddefault}{\updefault}{\color[rgb]{0,0,0}$x_1$}%
}}}}
\put(2851,-1186){\makebox(0,0)[lb]{\smash{{\SetFigFont{9}{10.8}{\rmdefault}{\mddefault}{\updefault}{\color[rgb]{0,0,0}$x_2$}%
}}}}
\put(4501,-1186){\makebox(0,0)[lb]{\smash{{\SetFigFont{9}{10.8}{\rmdefault}{\mddefault}{\updefault}{\color[rgb]{0,0,0}$\frac{1+x_2}{x_1}$}%
}}}}
\put(6451,-1186){\makebox(0,0)[lb]{\smash{{\SetFigFont{9}{10.8}{\rmdefault}{\mddefault}{\updefault}{\color[rgb]{0,0,0}$x_2$}%
}}}}
\put(8101,-1186){\makebox(0,0)[lb]{\smash{{\SetFigFont{9}{10.8}{\rmdefault}{\mddefault}{\updefault}{\color[rgb]{0,0,0}$\frac{1+x_2}{x_1}$}%
}}}}
\put(9826,-1186){\makebox(0,0)[lb]{\smash{{\SetFigFont{9}{10.8}{\rmdefault}{\mddefault}{\updefault}{\color[rgb]{0,0,0}$\frac{1+x_1+x_2}{x_1 x_2}$}%
}}}}
\put(1051,-3586){\makebox(0,0)[lb]{\smash{{\SetFigFont{9}{10.8}{\rmdefault}{\mddefault}{\updefault}{\color[rgb]{0,0,0}$x_2$}%
}}}}
\put(2851,-3586){\makebox(0,0)[lb]{\smash{{\SetFigFont{9}{10.8}{\rmdefault}{\mddefault}{\updefault}{\color[rgb]{0,0,0}$x_1$}%
}}}}
\put(4501,-3586){\makebox(0,0)[lb]{\smash{{\SetFigFont{9}{10.8}{\rmdefault}{\mddefault}{\updefault}{\color[rgb]{0,0,0}$\frac{1+x_1}{x_2}$}%
}}}}
\put(6451,-3586){\makebox(0,0)[lb]{\smash{{\SetFigFont{9}{10.8}{\rmdefault}{\mddefault}{\updefault}{\color[rgb]{0,0,0}$x_1$}%
}}}}
\put(8101,-3586){\makebox(0,0)[lb]{\smash{{\SetFigFont{9}{10.8}{\rmdefault}{\mddefault}{\updefault}{\color[rgb]{0,0,0}$\frac{1+x_1}{x_2}$}%
}}}}
\put(9751,-3586){\makebox(0,0)[lb]{\smash{{\SetFigFont{9}{10.8}{\rmdefault}{\mddefault}{\updefault}{\color[rgb]{0,0,0}$\frac{1+x_1+x_2}{x_1 x_2}$}%
}}}}
\put(3676,-136){\makebox(0,0)[lb]{\smash{{\SetFigFont{12}{14.4}{\rmdefault}{\mddefault}{\updefault}{\color[rgb]{0,0,0}$\mu_1$}%
}}}}
\put(7276,-211){\makebox(0,0)[lb]{\smash{{\SetFigFont{12}{14.4}{\rmdefault}{\mddefault}{\updefault}{\color[rgb]{0,0,0}$\mu_2$}%
}}}}
\put(9601,-1936){\makebox(0,0)[lb]{\smash{{\SetFigFont{12}{14.4}{\rmdefault}{\mddefault}{\updefault}{\color[rgb]{0,0,0}$\mu_1$}%
}}}}
\put(7276,-2686){\makebox(0,0)[lb]{\smash{{\SetFigFont{12}{14.4}{\rmdefault}{\mddefault}{\updefault}{\color[rgb]{0,0,0}$\mu_2$}%
}}}}
\put(3601,-2686){\makebox(0,0)[lb]{\smash{{\SetFigFont{12}{14.4}{\rmdefault}{\mddefault}{\updefault}{\color[rgb]{0,0,0}$\mu_1$}%
}}}}
\end{picture}%

%% file: Y-A2.pstex_t
\begin{picture}(0,0)%
\includegraphics{Y-A2.pstex}%
\end{picture}%
\setlength{\unitlength}{1816sp}%
\begingroup\makeatletter\ifx\SetFigFont\undefined%
\gdef\SetFigFont#1#2#3#4#5{%
  \reset@font\fontsize{#1}{#2pt}%
  \fontfamily{#3}\fontseries{#4}\fontshape{#5}%
  \selectfont}%
\fi\endgroup%
\begin{picture}(9274,3748)(1036,-3695)
\put(2851,-3586){\makebox(0,0)[lb]{\smash{{\SetFigFont{9}{10.8}{\rmdefault}{\mddefault}{\updefault}{\color[rgb]{0,0,0}$y_1$}%
}}}}
\put(1051,-1186){\makebox(0,0)[lb]{\smash{{\SetFigFont{9}{10.8}{\rmdefault}{\mddefault}{\updefault}{\color[rgb]{0,0,0}$y_1$}%
}}}}
\put(2851,-1186){\makebox(0,0)[lb]{\smash{{\SetFigFont{9}{10.8}{\rmdefault}{\mddefault}{\updefault}{\color[rgb]{0,0,0}$y_2$}%
}}}}
\put(4501,-1186){\makebox(0,0)[lb]{\smash{{\SetFigFont{9}{10.8}{\rmdefault}{\mddefault}{\updefault}{\color[rgb]{0,0,0}$y_1^{-1}$}%
}}}}
\put(9826,-1186){\makebox(0,0)[lb]{\smash{{\SetFigFont{9}{10.8}{\rmdefault}{\mddefault}{\updefault}{\color[rgb]{0,0,0}$\frac{1}{y_2(1+y_1)}$}%
}}}}
\put(9751,-3586){\makebox(0,0)[lb]{\smash{{\SetFigFont{9}{10.8}{\rmdefault}{\mddefault}{\updefault}{\color[rgb]{0,0,0}$\frac{1+y_2}{y_1 y_2}$}%
}}}}
\put(4501,-3586){\makebox(0,0)[lb]{\smash{{\SetFigFont{9}{10.8}{\rmdefault}{\mddefault}{\updefault}{\color[rgb]{0,0,0}$y_2^{-1}$}%
}}}}
\put(1051,-3586){\makebox(0,0)[lb]{\smash{{\SetFigFont{9}{10.8}{\rmdefault}{\mddefault}{\updefault}{\color[rgb]{0,0,0}$y_2$}%
}}}}
\put(6001,-1186){\makebox(0,0)[lb]{\smash{{\SetFigFont{9}{10.8}{\rmdefault}{\mddefault}{\updefault}{\color[rgb]{0,0,0}$y_2(1+y_1)$}%
}}}}
\put(7876,-1186){\makebox(0,0)[lb]{\smash{{\SetFigFont{9}{10.8}{\rmdefault}{\mddefault}{\updefault}{\color[rgb]{0,0,0}$\frac{1+y_2+y_1 y_2}{y_1}$}%
}}}}
\put(7876,-3586){\makebox(0,0)[lb]{\smash{{\SetFigFont{9}{10.8}{\rmdefault}{\mddefault}{\updefault}{\color[rgb]{0,0,0}$\frac{y_1}{1+y_2+y_1 y_2}$}%
}}}}
\put(6301,-3586){\makebox(0,0)[lb]{\smash{{\SetFigFont{9}{10.8}{\rmdefault}{\mddefault}{\updefault}{\color[rgb]{0,0,0}$\frac{y_1 y_2}{1+y_2}$}%
}}}}
\put(3751,-286){\makebox(0,0)[lb]{\smash{{\SetFigFont{12}{14.4}{\rmdefault}{\mddefault}{\updefault}{\color[rgb]{0,0,0}$\mu_1$}%
}}}}
\put(7351,-286){\makebox(0,0)[lb]{\smash{{\SetFigFont{12}{14.4}{\rmdefault}{\mddefault}{\updefault}{\color[rgb]{0,0,0}$\mu_2$}%
}}}}
\put(9526,-1861){\makebox(0,0)[lb]{\smash{{\SetFigFont{12}{14.4}{\rmdefault}{\mddefault}{\updefault}{\color[rgb]{0,0,0}$\mu_1$}%
}}}}
\put(3751,-2686){\makebox(0,0)[lb]{\smash{{\SetFigFont{12}{14.4}{\rmdefault}{\mddefault}{\updefault}{\color[rgb]{0,0,0}$\mu_1$}%
}}}}
\put(7351,-2686){\makebox(0,0)[lb]{\smash{{\SetFigFont{12}{14.4}{\rmdefault}{\mddefault}{\updefault}{\color[rgb]{0,0,0}$\mu_2$}%
}}}}
\end{picture}%